\documentclass{amsart}
\usepackage{geometry}                % See geometry.pdf to learn the layout options. There are lots.
\usepackage{graphicx}
\usepackage{amssymb}
\usepackage{epstopdf,amsmath}

\newtheorem{prop}{Proposition}[section]
\newtheorem{lemma}[prop]{Lemma}
\newtheorem{cor}[prop]{Corollary}
\newtheorem{theorem}[prop]{Theorem}
\newtheorem*{theorem*}{Theorem}
\newtheorem{conj}[prop]{Conjecture}

\newtheorem{exam}[prop]{Example}
\newtheorem{ques}[prop]{Question}
\newtheorem{defn}[prop]{Definition}
\newtheorem{heur}[prop]{Heuristic}

\newcommand{\CC}{\mathbb{C}}
\newcommand{\eps}{\epsilon}

\newcommand{\ZZ}{\mathbb{Z}}
\newcommand{\RR}{\mathbb{R}}
\newcommand{\TT}{\mathbb{T}}
\newcommand{\supp}{\textrm{supp}}

\title{Decoupling estimates in Fourier analysis}

\author{Larry Guth}
\begin{document}

\begin{abstract} Decoupling is a recent development in Fourier analysis, which has applications in harmonic analysis, PDE, and number theory.  We survey some applications of decoupling and some of the ideas in the proof.  This survey is aimed at a general mathematical audience.  It is based on my 2022 ICM talk.
\end{abstract}

\maketitle

\section{Introduction}

Decoupling is a recent development in Fourier analysis, which has applications in harmonic analysis, PDE, and number theory.  To put it in context, let's start by recalling some basic ideas of Fourier analysis.  In Fourier analysis, we represent a function as a Fourier series or Fourier integral.  For instance, if $f: \RR^n \rightarrow \CC$ is a reasonably nice function, then we can write it as a Fourier integral:

\begin{equation} \label{fourierintegral} f(x) = \int_{\RR^n} \hat f(\omega) e^{2 \pi i \omega \cdot x} dx. \end{equation}

\noindent Here $\omega \cdot x$ is the dot product of $\omega$ and $x$, which we will also abbreviate as just $\omega x$.

Here are a couple reasons that it's useful to represent a function $f$ using a Fourier series or integral.  First, the functions $e^{2 \pi i \omega x}$ are eigenfunctions for the partial derivative operators $\partial x_j$.  This makes the Fourier representation interact well with partial derivatives, and it helps to study PDE.  Second, the functions $e^{2 \pi i \omega \cdot x}$ are eigenfunctions of the translation operator $T_v$ defined by $T_v f(x) = f(x + v)$.  This makes the Fourier representation useful in problems involving the translation structure of $\RR^n$, including problems in additive number theory.

But there is also a serious downside to representing a function $f$ as a Fourier series/integral.  To evaluate $f(x)$, we have to do an integral or a sum with many terms.  It often happens that the terms have various phases in the complex plane, and it's difficult to tell what happens when we add them all up.  In general, given some information about $\hat f$, it can be difficult to determine what that information has to say about $f$.  We will see some longstanding open questions of this flavor below.

Decoupling is helpful for estimating $\| f \|_{L^p}$ in terms of information about $\hat f$.  Now $ \| f \|_{L^2}$ is directly related to $\hat f$ because of orthogonality:  Plancherel's theorem states that

\begin{equation} \label{plancherel} \| f \|_{L^2} = \| \hat f \|_{L^2}.
\end{equation}

\noindent But for other values of $p$, it is much harder to connect $\| f \|_{L^p}$ with information about $\hat f$.  

Estimates for $\| f \|_{L^p}$ for $p \not= 2$ occur often in harmonic analysis, PDE, and analytic number theory.
You may wonder, if we have a good understanding of $ \| f \|_{L^2}$, what more do we learn by understanding $\| f \|_{L^p}$ for other values of $p$.  I like to think of this question in terms of super-level sets.  Define the super-level set $U_\lambda(f)$ by

\begin{equation} \label{defUlambda} U_\lambda(f) := \{ x : |f(x)| > \lambda \}. \end{equation}

\noindent We denote the volume of a set $U$ by $|U|$.  If we know $\| f \|_{L^p}$ for every $p$, we typically get accurate estimates for $|U_\lambda(f)|$ for every $\lambda$, which gives us basically all the possible information about how "big" the function $f$ is.  But if we only know $ \| f \|_{L^2}$, we get only limited information about $|U_\lambda(f)|$.

Other motivations for studying $\| f \|_{L^p}$ come from applications in PDE and analytic number theory.  In non-linear PDE, bounds involving $\| f \|_{L^p}$ are important for understanding how close a solution to a non-linear PDE is to a solution of a corresponding linear PDE.  In analytic number theory, the number of solutions to certain diophantine systems is equal to $\int |f|^p$ for a well-chosen function $f$ and exponent $p$.  These are just a couple samples among many applications for estimating $\| f \|_{L^p}$.

Decoupling is a new tool for estimating $\| f \|_{L^p}$ in terms of Fourier analytic information about $f$.  It was first formulated by Wolff in \cite{W4}, where he was able to prove sharp estimates for large values of $p$.  In \cite{BD}, Bourgain and Demeter proved sharp decoupling estimates for all $p$.  This breakthrough has led to solutions to problems in harmonic analysis that had seemed far out of reach a decade ago.  

In the next two subsections, we will introduce two main areas where decoupling has had an impact.  We will give examples of hard open problems and also examples of problems that were solved using decoupling.

\subsection{Restriction theory} \label{secrestr}

The Fourier representation of a function $f: \RR^n \rightarrow \CC$ is

$$f(x) = \int_{\RR^n} \hat f(\omega) e^{2 \pi i \omega \cdot x} d \omega. $$

There are two basic estimates connecting the $L^p$ norms of $f$ and the $L^p$ norms of $\hat f$:

\begin{itemize}

\item Orthogonality: $\| f \|_{L^2} = \| \hat f \|_{L^2}$.

\item Triangle inequality: $\| f \|_{L^\infty} \le \| \hat f \|_{L^1}$.

\end{itemize}

\noindent Interpolating between these gives the Hausdorff-Young inequality

\begin{equation} \label{HausdorffYoung} \| f \|_{L^p} \le \| \hat f \|_{L^q} \textrm{ if $1 \le p \le 2$ and $\frac{1}{q} = 1 - \frac{1}{p}$} \end{equation}

\noindent These are all of the $L^p$ type estimates for the Fourier transform operator.

If $\hat f$ is supported in a subset $\Omega \subset \RR^n$, we can write

$$f(x) = \int_{\Omega} \hat f(\omega) e^{2 \pi i \omega \cdot x} d \omega. $$

\noindent Restriction theory studies how the geometry of $\Omega$ relates to properties of $f$ such as $\| f \|_{L^p}$.  One of the most interesting cases is when $\hat f$ is supported in a compact submanifold $S \subset \RR^n$.  
In this case, the Fourier representation of $f$ has the form

\begin{equation} \label{defextension} f(x) = \int_S a(\omega) e^{2 \pi i \omega x} d \mu_S(\omega),\end{equation}

\noindent where $d \mu_S$ is the surface area measure of $S$.

Stein proposed studying $L^p$ estimates of the form 

\begin{equation} \label{extensionproblem} \| f \|_{L^p(\RR^n)}  \le C \| a \|_{L^q(S)} \end{equation}

\noindent Stein made the remarkable discovery that the estimates for the operator $E_S$ depend on the geometry of $S$.  If $S$ is a flat disk, then the only estimate of the form \ref{extensionproblem} is the triangle inequality: $\| f \|_{L^\infty} \le \| a \|_{L^1(S)}$.  But if $S$ is a curved surface, then there are more inequalities.  One central problem in the field is to understand all the $L^p$ inequalities of form \ref{extensionproblem} when $S$ is a curved hypersurface, like a paraboloid.  Let us write $P$ for the truncated paraboloid

\begin{equation} \label{defP} P := \{ \omega \in \RR^n | \omega_n = \sum_{j=1}^{n-1} \omega_j^2 \textrm{ and } \sum_{j=1}^{n-1} \omega_j^2 \le 1 \} \end{equation}

In this case, the Fourier representation of $f$ is

\begin{equation} \label{deffP} f(x) = \int_P a(\omega) e^{2 \pi i \omega x} d \mu_P(\omega).\end{equation}

\begin{exam} \label{bump} Suppose $a(\omega) = 1$ on $P$, and $f$ is given by (\ref{deffP}).

First note that $f(0) = \int_P d \mu_P$ is equal to the area of $P$, which is $\sim 1$.  When $x$ is large, there is a lot of cancellation in the integral (\ref{deffP}) coming from rapid oscillation of the function $e^{2 \pi i \omega x}$ as $\omega$ varies over $P$.  This effect can be estimated accurately using stationary phase, and one finds that 

$$ |f(x)|  \lesssim |x|^{- \frac{n-1}{2}}.$$

\noindent This bound is sharp for most $x$.  Therefore $\| f \|_{L^p(\RR^n)} < \infty$ if and only if $p > \frac{2n}{n-1}$.

\end{exam}

Stein conjectured that the same $L^p$ bounds hold whenever $|a(\omega)| \le 1$ for all $\omega$.

\begin{conj} \label{conjrestriction} (Restriction conjecture, \cite{Ste}) Suppose that $f$ has the form (\ref{deffP}) and that $|a(\omega)| \le 1$ for all $\omega \in P$.  If $p > \frac{2n}{n-1}$, then 

$$ \| f \|_{L^p(\RR^n)} \le C(p,n). $$

\end{conj}

Notice that the hypothesis that $f$ has the form (\ref{deffP}) with $|a(\omega)| \le 1$ for all $\omega$ is a hypothesis about $\hat f$.  The restriction conjecture asks what this information about $\hat f$ tells us about $\| f \|_{L^p}$.  
The 2-dimensional case of Conjecture \ref{conjrestriction} was proven by Fefferman in \cite{F1}.  But for dimension $n \ge 3$, the conjecture remains open after intensive work by many people.  In Section \ref{seckakeya}, we will discuss some reasons the problem is so difficult.  

In Conjecture \ref{conjrestriction}, we considered the bound $\| a(\omega) \|_{L^\infty} \le 1$.  Bounds of the form $\| a \|_{L^q(P)}$ are also interesting for other $q$.  The case $q=2$ is the most important, and it was completely worked out by  Strichartz \cite{Str} following work by Tomas and Stein.  It has turned out to be important in PDE.  It reads as follows.

\begin{theorem} \label{strichartz} (Strichartz inequality, \cite{Str}) Suppose that $f$ has the form (\ref{deffP}.  If $p \ge \frac{2 (n+1)}{n-1}$, then 

$ \| f \|_{L^p(\RR^n)} \le C(n) \| a(\omega) \|_{L^2(P)}. $

\end{theorem}

This theorem plays an important role in the study of the Schrodinger equation.  Recall that the linear Schodinger equation for a function $u(x,t)$ with $x \in \RR^{d}$ and $t \in \RR$ is 

\begin{equation} \label{schrodeq} \partial_t u = i \sum_{j=1}^d \partial^2_{x_j} u \end{equation}

If $u$ obeys the linear Schrodinger equation, then $\hat u$ is a distribution supported on the paraboloid, and so the Strichartz estimate can be used to understand $\| u \|_{L^p}$.  Theorem \ref{strichartz} tells us that for any solution of the linear Schrodinger equation (\ref{schrodeq}) with initial data $u(x,0) = u_0(x)$, 

\begin{equation} \label{strichschrod} \| u \|_{L^{\frac{2(d+2)}{d}}(\RR^d \times \RR)} \le C \| u_0 \|_{L^2(\RR^d)}. \end{equation}

This theorem has played a central role in PDE, especially in non-linear PDE.  The $L^2$ norm on the right-hand side is important in PDE because $\| u_0 \|_{L^2} = \| \hat u_0 \|_{L^2}$ and also $\| u_0 \|_{L^2} = \| u( y, t) \|_{L^2_y}$ for every $t$.  In non-linear PDE, it leads to sharp estimates about when the solution to a non-linear PDE is close to the solution of the corresponding linear PDE.

The Strichartz estimate describes a spreading out effect.  To get a sense of it, first suppose that $u_0$ is a smooth bump concentrated on a ball in space time.  As $t$ increases, the function $u(x,t)$ spreads out and gets smaller.  As it does so, $\int_{\RR^d} |u(x,t)|^2 dx$ remains constant, and $\int_{\RR^d} |u(x,t)|^p dx$ gets smaller for any $p > 2$.  Because of this spreading out effect, $\int_{\RR^d \times \RR} |u(x,t)|^{\frac{2(d+2)}{d}} dx dt$ is finite.

The exponent $\frac{2(d+2)}{d}$ is the only exponent for which (\ref{strichschrod}) holds.  To see what is special about this exponent,  it helps me to translate the Strichartz estimate into an estimate for superlevel sets.  Let $U_\lambda(u) := \{(x,t) \in \RR^d \times \RR: |u(x,t)| > \lambda \}$.  The Strichartz inequality implies that if $\| u_0 \|_{L^2(\RR^d)} = 1$, then

$$ | U_\lambda(u) | \le C \lambda^{-\frac{2(d+2)}{d}} . $$

\noindent This estimate is sharp: for any choice of $\lambda$ we can find initial data $u_0$ with $\| u_0 \|_{L^2(\RR^d)} = 1$ so that the solution of the Schrodinger equation has $|U_\lambda(u)| \ge c \lambda^{-\frac{2(d+2)}{d}}$.

It's also worth mentioning that the choice of the paraboloid in this discussion is just one interesting example.  There are similar theorems and conjectures for other surfaces, such as the sphere and the cone, and these help to study other PDE, such as the Laplace eigenfunction equation $\triangle u = \lambda u$ and the wave equation.

One striking application of decoupling involves Strichartz estimates on flat tori.  The Schrodinger equation makes sense on any Riemannian manifold, and for each manifold we can ask for the best inequality in the spirit of (\ref{strichschrod}).  Understanding the Strichartz estimates on closed manifolds is extremely difficult.  It is known that different closed manifolds behave quite differently from each other -- for example, round spheres behave differently from flat tori.  But very few examples are understood.  Before decoupling, sharp Strichartz estimates were only known for $S^1$ and $S^1 \times S^1$ (by Bourgain in the 90s \cite{B11}) and $S^3$ (by Burq-Gerard-Tzvetkov \cite{BGT}).  In all these examples, the value of the exponent $p$ is an even integer, and we will discuss in Subsection \ref{analnumb} why this is important.

The simplest flat torus in the unit cube torus $\RR^d / \ZZ^d$.  A solution to the Schrodinger equation on the unit cube torus is just a solution $u(x,t)$ on $\RR^d \times \RR$ which is $\ZZ^d$-periodic in the $x$ variable.  Any such solution can be written in the form

\begin{equation} \label{defusp} u(x,t) = \sum_{n \in \ZZ^d} a_n e^{2 \pi i (n \cdot x + |n|^2 t)}. \end{equation}

\noindent Notice that this Fourier representation is analogous to (\ref{deffP}), except that the integral in ({\ref{deffP}) is replaced by a sum.  
We say that $u$ has ``frequency at most $N$" if the coefficients $a_n$ are supported in the cube $Q_N : = \{ (n_1, ...., n_d) \in \ZZ^d: |n_j| \le N \textrm{ for all } j \}$.  

\begin{exam} \label{fundtorus}  Suppose that $u$ is given by (\ref{defusp}) where $a_n = 1$ if $n \in Q_N$ and $a_n = 0$ otherwise.  In other words

$$ u(x,t) = \sum_{n \in Q_N}  e^{2 \pi i (n \cdot x + |n|^2 t)}. $$

First note that $u(0,0) = |Q_N| \sim N^d$.  We have $|u(x,t)| \sim N^d$ when $|x| \le \frac{1}{10 d N}$ and $|t| \le \frac{1}{10 d N^2}$, because then each term in the sum is almost 1.  As $x$ and $t$ increase, we get cancellation in the sum coming from oscillations in $e^{2 \pi i (n \cdot x + |n|^2 t)}$.  So far this behavior is similar to Example \ref{bump}.  

However, in the torus case, $|u(x,t)|$ is also large when $(x,t)$ lies near to a rational point of the form $(\frac{p_1}{q}, ... , \frac{p_d}{q}, \frac{p_t}{q})$.  Taking account of all these peaks near rational points, it turns out that $U_\lambda (u) \cap [0,1]^{d+1}$ has volume $\sim N^{d+2} \lambda^{-\frac{2(d+2)}{d}}$ for all $\lambda$ in the range $N^{d/2} \le \lambda \le N^d$.  (This range includes all interesting values of $\lambda$.)

\end{exam}

A natural analogue of the restriction conjecture in the periodic setting would say

\begin{conj} \label{discreterestriction} Suppose that $u$ is given by (\ref{defusp}) and that $|a_n| \le 1$ for all $n \in Q_N$ and $a_n = 0$ for $n \notin Q_N$.  Then $|U_\lambda(u) \cap [0,1]^{d+1} | \le C(d, \eps) N^{d+2 + \eps} \lambda^{-\frac{2(d+2)}{d}}$ for all $\lambda$ in the range $N^{d/2} \le \lambda \le N^d$.
\end{conj}

\noindent This conjecture used to sound to me just as hard as the restriction conjecture or maybe harder.  The set up is similar.  And Example \ref{fundtorus} in this periodic setting is more intricate and complex than Example \ref{bump} in the setting of the original restriction conjecture.  However, Bourgain and Demeter proved this conjecture as a corollary of their sharp Strichartz estimate on tori.  This theorem is one of the first applications of decoupling.

\begin{theorem} \label{strichtorus} (Bourgain and Demeter \cite{BD}) Suppose that $u$ is given by (\ref{defusp}) and that $a_n$ is supported in $Q_N$.  Then

\begin{equation} \label{perstrich} \| u \|_{L^{\frac{2(d+2)}{d}}([0,1]^{d+1})} \le  C(d, \eps) N^\eps \| a_n \|_{\ell^2}. \end{equation}

\end{theorem}

Notice that if $u_0(x) = u(x,0)$, then $\| a_n \|_{\ell^2} = \|u_0 \|_{L^2([0,1]^d)}$, so this inequality is very similar to the Strichartz inequality for the Schrodinger equation on $\RR^d$ recorded in (\ref{strichschrod}).

To finish this section, let us try to roughly indicate why the Strichartz inequality on the torus is much harder than the Strichartz inequality on $\RR^d$.  Recall that the Strichartz inequality encodes a spreading out effect.  First imagine a solution $u(x,t)$ on Euclidean space, and suppose that the initial data $u_0$ is concentrated in a very small ball.  As time increases, the solution $u(x,t)$ spreads out.  At a small time $t_0$, the solution is spread over a unit ball.  In Euclidean space, it can continue to spread out in all directions indefinitely.  The proof of Strichartz estimates this effect in a quantitative way.  

Now let $u_P$ be the solution on the torus with the same initial data $u_0$.  The function $u_P$ is given by periodizing $u$:

\begin{equation} \label{uP} u_P(x,t) = \sum_{z \in \ZZ^d} u(x + z, t). \end{equation}

\noindent For times up to $t_0$, $u(x,t)$ is supported on a unit ball in $x$ variable, and so $u_P(x,t) = u(x,t)$.  But beyond this time, $u(x,t)$ is spread over a much bigger ball, and there are many non-zero terms in the sum
(\ref{uP}).  If we visualize $u_P(x,t)$, the solution starts to wrap around the torus.  Different pieces of the solution, which have traveled around the torus in different ways, get added up, and we have to prove that there is a lot of cancellation in that sum.  

Before decoupling, Theorem \ref{strichtorus} was known for $d=1,2$ only because of a connection with number theory.  In the next section, we describe some connections between Fourier analysis and number theory and we will flesh this out.

\subsection{Analytic number theory} \label{analnumb}

When $p$ is an even integer, $L^p$ estimates have a special interpretation which connects them with problems in additive number theory.

Suppose that $A \subset \ZZ^d$ is a finite set.  We define $E_s(A)$ (the additive $s$-energy of $A$) by

\begin{equation} \label{saddenergy} E_s(A) := \# \{ (a_1, ..., a_s, b_1, ..., b_s) \in A^{2s} : a_1 + ... + a_s = b_1 + ... + b_s \}. \end{equation}

For each $A$, we can also define a function $f_A(x)$ with Fourier series

\begin{equation} \label{fA} f_A(x) = \sum_{a \in A} e^{2 \pi i a \cdot x}. \end{equation}

The function $f_A: \RR^d \rightarrow \CC$ is $\ZZ^d$ periodic because $A \subset \ZZ^d$ and so each function $e^{2 \pi i a \cdot x}$ is $\ZZ^d$-periodic.

\begin{lemma} \label{eventrick} For any finite set $A \subset \ZZ^d$,

$$ \int_{[0,1]^d} |f_A(x)|^{2s} dx = E_s(A). $$
\end{lemma}

\begin{proof} [Proof sketch] We expand out the integral on the LHS.

$$ \int_{[0,1]^d} |f_A(x)|^{2s} dx = \int_{[0,1]^d} f_A^s \bar f_A^s dx = \int_{[0,1]^d} \sum_{a_1, ..., a_s, b_1, ..., b_s \in A} e^{2 \pi i (a_1 + ... + a_s - b_1 - ... - b_s) x} dx.$$

Now if $m \in \ZZ^d$, then $\int_{[0,1]^d} e^{2 \pi i m \cdot x} dx$ is 1 if $m = 0$ and 0 otherwise.  And so the only terms that contribute to the integral above are terms where $a_1 + ... + a_s - b_1 - ... - b_s = 0$.  So the last integral is $E_s(A)$.

\end{proof}

For instance, if $A_{k,N} := \{ 1^k, 2^k, ..., N^k \} \subset \ZZ$ then

\begin{equation} E_s(A_{k,N}) = \# \textrm{ of solutions to } a_1^k + ... + a_s^k = b_1^k + ... + b_s^k, \textrm{ with $a_j, b_j \in \ZZ, 1 \le a_j, b_j \le N$}. \end{equation}

In this case, the relevant function $f$ is 

\begin{equation} \label{fkN} f_{k,N} (x) = \sum_{a=1}^N e^{2 \pi i a^k x}, \end{equation}

\noindent and Lemma \ref{eventrick} tells us that 

\begin{equation} \label{eventrickkN} \int_0^1 |f_{k,N}(x)|^{2s} dx = E_s(A_{k,N}). \end{equation}

Lemma \ref{eventrick} tells us that a certain $L^p$ norm is equal to the number of solutions to a certain diophantine equation.  The lemma
is useful in both directions.  If we know something about the number of solutions to the diophantine equation, then we can get information about the $L^p$ norm.  If we know something about the $L^p$ norm, then we can get information about the number of solutions to the diophantine equation.

For instance, consider the diophantine equation $a_1^2 + a_2^2 = b_1^2 + b_2^2$, with $a_i, b_i$ between 1 and $N$.  First let us estimate the number of solution directly.  Rearranging we get $a_1^2 - b_1^2 = b_2^2 - a_2^2$, and factoring one side we see that 

$$ (a_1 + b_1) (a_1 - b_1) = b_2^2 - a_2^2. $$

\noindent If we fix $a_2, b_2$, then the number of $(a_1, b_1)$ solving this equation depends on the number of factors of $b_2^2 - a_2^2$.  Because of unique factorization, the number of different factors of an integer $M$ is fairly small, at most $C_\eps M^\eps$ for any $\eps > 0$.  Using this, we see that the number of integer solutions to $a_1^2 + a_2^2 = b_1^2 + b_2^2$ with $1 \le a_j, b_j \le N$ is at most $C_\eps N^{2 + \eps}$.  Lemma \ref{eventrick} tells us that the number of solutions is equal to $\int_0^1 |f_{2,N}(x)|^4 dx$, and so we conclude that this integral is bounded by $C_\eps N^{2 + \eps}. $

On the other hand, Weyl used the differencing method to give pointwise estimates for the function $f_{2,N}$.  These estimates imply that $\int_0^1 |f_{2,N}(x)|^4 dx \le C_\eps N^{2 + \eps}$ which then gives an analytic proof that the number of integer solutions to $a_1^2 + a_2^2 = b_1^2 + b_2^2$ with $1 \le a_j, b_j \le N$ is at most $C_\eps N^{2 + \eps}$.

Hardy and Littlewood made a conjecture that generalizes these estimates from squares to higher powers.

\begin{conj} \label{HypK*} (Hardy and Littlewood) For any $k \ge 2$, $E_k(A_{k,N}) \le C_\eps N^{k + \eps}$.  Equivalently,

$$ \int_0^1 |f_{k,N}(x)|^{2k} dx \le C_\eps N^{k + \eps}. $$

\end{conj}

This conjecture is open for all $k \ge 3$.  The Fourier series of $f_{3, N}$ is fairly simple to write down.  But it is very difficult to determine good bounds for the $L^p$ norms of $f_{3,N}$, or for the size of superlevel sets $U_{\lambda}(f_{3,N})$.  This is a classical and striking example of how difficult it is to read off information about $f(x)$ from information about its Fourier series.

On the other hand, there are cases when we can use Fourier analysis to estimate an $L^p$ norm and then use Lemma \ref{eventrick} to get a new estimate for the number of solutions to a diophantine equation.  One of the most interesting examples of this kind concerns Vinogradov's mean value theorem, which is a multivariable generalization of the functions we just considered.

Define

$$F_{k,N}(x_1, ..., x_k) = \sum_{a=1}^N e^{2 \pi i (a x_1 + a^2 x_2 + ... + a^k x_k)}. $$

By Lemma \ref{eventrick}, $\int_{[0,1]^k} |F_{k,N}(x)|^{2s} dx$ is equal to the number of solutions to the following diophantine system of equations:

$$ a_1^j + ... + a_s^j = b_1^j + ... + b_s^j \textrm{ for all } 1 \le j \le k, \textrm{ with } a_i, b_i \in \ZZ, 1 \le a_i, b_i \le N $$

Vinogradov \cite{V} studied the $L^p$ norms of $F_{k,N}$ in the 1930s.  He was able to prove sharp estimates for $\| F_{k,N} \|_{L^p}$ for sufficiently large $p$.  He used these bounds to greatly improve the estimates for Weyl sums and Waring's problem in large degree, and also to improve the bounds on the zero-free region of the Riemann zeta function.  Vinogradov's argument cleverly exploited both sides of Equation (\ref{eventrick}): some parts of the argument directly count the number of solutions to some diophantine systems in the variables $a_i, b_i$, and other parts of the argument estimate integrals in the $x$ variable.  Some important ideas in the proof of decoupling are related to Vinogradov's argument, and we will discuss this more in Section \ref{subsecindplustrans}.  

In the last decade, mathematicians have proven estimates for $\| F_{k,N} \|_{L^p}$ that are sharp up to factors of $C(k, \eps) N^\eps$ for every $k$ and $p$.  As a corollary, we get estimates for the number of solutions to the Vinogradov system that are sharp up to a factor $C(k, \eps) N^\eps$.  

\begin{theorem} \label{vino} (\cite{BDG}, \cite{Woo2}, \cite{Woo3})

$$ \| F_{k,N} \|_{L^p([0,1]^k)} \le C(k, \eps) N^\eps \left( N^{1/2} + N^{1 - \frac{k(k+1)}{2p}} \right). $$

\end{theorem}

The proof in \cite{BDG} uses decoupling and the proof in \cite{Woo3} uses the method of efficient congruencing.  (Historically, Wooley developed efficient congruencing starting in the 90s, cf. \cite{Woo1}.  He improved Vinogradov's estimates and gave sharp estimates for $k=3$ in \cite{Woo2}.  Then \cite{BDG} used decoupling to prove Theorem \ref{vino} and immediately afterwards, \cite{Woo3} used efficient congruencing to give a different proof of Theorem \ref{vino}.)

Both \cite{BDG} and \cite{Woo3} are quite technical.  Recently, Guo-Li-Yung-Zorin-Kranich \cite{GLYZ} gave a dramatically simpler proof of Theorem \ref{vino}, combining some of the features of \cite{BDG} and \cite{Woo3} with some new clarifying ideas.  Their paper is ten pages long and essentially self-contained.

Lemma \ref{eventrick} is a special trick for understanding $L^p$ norms when $p$ is an even integer.  This even integer trick also plays an important role in the problems we discussed in Section \ref{secrestr}.  In \cite{F1}, Fefferman used a version of the even integer trick to prove Conjecture \ref{conjrestriction} in dimension $n=2$.  The $L^p$ exponent in Conjecture \ref{conjrestriction} is $p = \frac{2n}{n-1}$, which is an even integer when $n=2$ but not for any $n \ge 3$.  In the early 90s, in \cite{B11}, Bourgain used the even integer trick to prove sharp periodic Strichartz estimates when $d=1, 2$ (the cases $d=1,2$ in Theorem \ref{strichtorus}).  The exponent in the Strichartz estimate is $\frac{2(d+2)}{d}$, which is an even integer when $d=1,2$, but not for any $d > 2$.  Another important problem in this circle is Montgomery's conjecture about the $L^p$ norms of Dirichlet polynomials.  When $p$ is an even integer, Montgomery gave sharp estimates for the relevant $L^p$ norms in just a page (cf. \cite{Mo}).  But giving a sharp estimate for any other value of $p$ is a major open problem.  This might help explain why, even though Theorem \ref{strichtorus} was already known in dimensions $d=1, 2$, it still seemed far out of reach to prove it for any other $d$.

Before decoupling, the situation concerning periodic Strichartz estimates in dimensions 1,2 was rather curious.  The periodic Strichartz estimate can be considered as a result in PDE, resolving a problem of mathematical physics.  But the proof depended on number theory facts, such as unique factorization.  The decoupling proof of Theorem \ref{strichtorus} is purely analytic -- with no input from number theory.  The argument can then recover some of the number theory that went into the original proof.  The relevant number theory estimates are not that difficult, but proving them by analysis is still interesting.  Building on this, Bourgain and Demeter began to work on Vinogradov's mean value theorem in \cite{BD2}, eventually leading to Theorem \ref{vino} and new results in number theory.

Theorem \ref{vino} leads to improved bounds for Waring's problem on the number of ways to write an integer as a sum of $k^{th}$ powers and the related problem of Weyl sums.  Other applications of decoupling have led to incremental improvements in other classical problems of analytic number theory such as the Lindelof hypothesis \cite{B6}.  Guo-Zhang \cite{GZ} and Guo-Zorin-Kranich \cite{GZK} have extended Theorem \ref{vino} to more complex systems of diophantine equations, introduced in number theory by Arkhipov-Chubarikov-Karatsuba \cite{ACK}.

\subsection{Influence of the proof}

Besides the new results, the method of proof of decoupling has had a big influence on the field.  There is a classical toolbox in harmonic analysis with tools like orthogonality, integration by parts, and Holder's inequality.  For hard problems in this area, such as the restriction conjecture, people who have worked a lot on them generally feel that this set of classical tools is not sufficient to understand the problem.  Over the last 25 years, mathematicians have brought into play ideas from other areas in order to attack some of these hard problems.  For instance, Wolff (\cite{W5} and \cite{W4}) brought in ideas from combinatorial geometry and topology, Bourgain (\cite{B5}) brought in ideas from combinatorial number theory, and Dvir (\cite{D}) brought in ideas from error-correcting codes and algebraic geometry.  In contrast to these developments, the proof of decoupling is based on the classical toolbox.  The most important idea in the proof is to take advantage of estimates at many different scales.  Using many different scales is also a classical idea in harmonic analysis.  But it is really striking how powerful it turns out to be in the context of decoupling.   I personally was shocked that it is possible to prove Theorem \ref{strichtorus} using only these tools.  The main goal of the article is to explore how combining information at many scales helps to prove theorems like Theorem \ref{strichtorus} and Theorem \ref{vino}.

\subsection{Outline of the rest of the article}

In Section 2, we will introduce the statement of decoupling.  In Section 3, we will begin to discuss multiscale arguments, and we will see how the statement of decoupling was carefully crafted to work well in such arguments.  In Section 4, we will discuss some ideas of the proof of decoupling.  

In Section 5, we will discuss the connection between the restriction problem and the Kakeya problem, and try to explain why the restriction problem seems to be so difficult.  Then we will discuss why decoupling turns out to be easier than restriction.  

In Section 6, we will survey some other applications of decoupling in harmonic analysis. 

In Section 7, we will discuss some limitations of the method, some frustrating aspects of the proof, and some open problems.

\section{The statement of decoupling}

Now that we have seen some applications of decoupling, we turn to the actual statement of decoupling.  The statement of decoupling was crafted carefully, and after we state it we will spend two sections digesting it and discussing some of the choices involved in the statement.

Suppose that $\Omega \subset \RR^n$ and that $\Omega$ is a disjoint union of subsets $\theta$: $ \Omega = \bigsqcup \theta. $
If $f: \RR^n \rightarrow \CC$ is a function, and $\hat f$ is supported in $\Omega$, then we can decompose $f = \sum_\theta f_\theta$ where $f_\theta$ is defined by

$$ f_\theta = \int_\theta \hat f(\omega) e^{2 \pi i \omega x} d \omega. $$

\noindent Decoupling has to do with the relationship between $L^p$ norm of $f$ and the $L^p$ norms of $f_\theta$ for the different $\theta$ in the decomposition $\Omega = \sqcup \theta$.

\begin{defn} \label{defdec} Suppose that $\Omega \subset \RR^n$ and $\Omega$ is a disjoint union of subsets $\theta$: $\Omega = \sqcup \theta$.  For each exponent $p$, we define the decoupling constant $D_p( \Omega = \sqcup \theta)$ to be the smallest constant so that for every function $f$ with $\hat f$ supported in $\Omega$,

\begin{equation} \label{decdef} \| f \|_{L^p(\RR^n)}^2 \le D_p(\Omega = \sqcup \theta)^2 \sum_{\theta} \| f_\theta \|_{L^p(\RR^n)}^2. \end{equation}

\end{defn}

\noindent If $p=2$, then orthogonality gives $\| f \|_{L^2}^2 = \sum_\theta \| f_\theta \|_{L^2}^2$, and so $D_2(\Omega = \sqcup \theta) =1$ for any decomposition $\Omega = \sqcup \theta$.  Decoupling theorems for higher $p$ are a kind of strengthening of orthogonality.  For $p > 2$, the value of $D_p(\Omega = \sqcup \theta)$ depends on the geometry of the decomposition.

As an example of a decomposition, first let $P$ denote the truncated parabola: 

$$P = \{ (\omega_1, \omega_2) \in \RR^2: \omega_2 = \omega_1^2, -1 \le \omega_1 \le 1 \}. $$

\begin{defn} \label{defPN}  For a large parameter $N$, we let $\Omega$ be the $N^{-2}$-neighborhood of $P$.  For $j=-N, ..., N$, we define 

$$\theta_j := \Omega \cap \{ \frac{j}{N} - \frac{1}{2N} \le \omega_1 \le \frac{j}{N} + \frac{1}{2N} \}. $$

\noindent Each $\theta_j$ is approximately a rectangular box of dimensions $N^{-2} \times N^{-1}$.

\noindent We have $\Omega = \sqcup_{j=1}^N \theta_j$, and we abbreviate this whole decomposition as $P_N$.

\end{defn}

We can now state our first decoupling theorem.

\begin{theorem} \label{decpar} (\cite{BD}) For each $\eps > 0$, for each $2 \le p \le 6$, $D_p(P_N) \le C_\eps N^\eps$.  

In other words, if $2 \le p \le 6$, and if $\hat f$ is supported in the $N^{-2}$-neighborhood of $P$, then

\begin{equation} \label{decpareq} \| f \|_{L^p(\RR^2)}^2 \le C_\eps N^\eps \sum_{j=1}^N \| f_{\theta_j} \|_{L^p(\RR^2)}^2. \end{equation}

\end{theorem}

This decoupling theorem can be applied to exponential sums, and it implies Theorem \ref{strichtorus} in the case $d=1$ and Theorem \ref{vino} in the case $k=2$.  Theorem \ref{strichtorus} for a $d$-dimensional torus follows from a decoupling theorem for the paraboloid in $\RR^{d+1}$, and Theorem \ref{vino} for higher $k$ follows from a decoupling theorem for the moment curve in $\RR^k$.

Let us see how this decoupling theorem leads to $L^p$ estimates for exponential sums.  This will help a little to digest the definition of $D_p$.  Suppose we start with an exponential sum using frequencies on the truncated parabola.  For $j=-N$ to $N$, we define the frequency $\omega_j = (\frac{j}{N}, \frac{j^2}{N^2}) \in P$, and we let $f$ be the exponential sum

$$f(x) = \sum_{j=-N}^N a_j e^{2 \pi i \omega_j x}. $$

If the $a_j$ were chosen randomly, then with high probability, we would have $|f(x)| \sim (\sum_j |a_j|^2)^{1/2}$ for most $x$.  In this random case, we would have
$ \| f \|_{L^p(B_R)} \sim (\sum_j |a_j|^2)^{1/2} |B_R|^{1/p}. $ So the best possible bound we could hope for has the form 

$$ \| f \|_{L^p(B_R)} \sim (\sum_j |a_j|^2)^{1/2} |B_R|^{1/p}. $$

Decoupling achieves such a bound up to a factor of $N^\eps$ when $2 \le p \le 6$ and $R$ is large enough.  This bound in turn implies Theorem \ref{strichtorus} for $d=1$ and Theorem \ref{vino} for $k=2$.  

Here is how to apply decoupling.  Note that the frequency $\omega_j$ lies in $\theta_j$.  In fact, if we write $f = \sum_j f_{\theta_j}$, then $f_{\theta_j} = a_j e^{2 \pi i \omega_j x}$.  Directly applying Theorem \ref{decpar} doesn't tell us anything because $\| f_{\theta_j} \|_{L^p(\RR^2)}$ is infinite.  But with a little technical work, one can prove that a similar estimate holds with $L^p$ norms on large balls instead of $L^p$ norms on the whole plane.  In particular, if $R \ge N^2$, then 

$$ \| f \|_{L^p(B_R)}^2 \le 100 D_p(P_N)^2 \sum_{j=-N}^N \| a_j e^{2 \pi i \omega_j x} \|_{L^p(B_R)}^2. $$

\noindent (The extra factor 100 comes from the technical work of passing from $\RR^2$ to $B_R$.)   If $p=6$, then we can plug in $D_6(P_N) \le C_\eps N^\eps$ and simplify everything to get

$$ \| f \|_{L^6(B_R)} \le C_\eps N^\eps (\sum_{j=1}^N |a_j|^2)^{1/2} |B_R|^{1/6}. $$

\noindent This bound matches the random example above up to the factor $C_\eps N^\eps$, and so in particular it is tight up to this factor.  This estimate is the periodic Strichartz estimate for $d=1$ and the Vinogradov mean value theorem for $k=2$.  

The definition of the decoupling constant $D_p$ was crafted partly to make this computation work.  This explains the squares in Definition \ref{defdec}.

\section{Induction on scales} \label{secindscales}

The definition of decoupling was crafted by Thomas Wolff in his work on local smoothing \cite{W4}.  He noticed that this definition is well suited for combining information from many scales.  The whole field of decoupling leans on this observation.  The first example of combining scales is the following lemma, which essentially appears in \cite{W4}.  

\begin{lemma} \label{submult} $D_p(P_{N_1 N_2}) \le D_p(P_{N_1}) D_p(P_{N_2}).$
\end{lemma} 

Let us first discuss why this is significant, and then we will sketch the proof.  If we iterate this lemma $k$ times, we get

\begin{equation} \label{submultk} D_p(P_{N_1^k}) \le D_p(P_{N_1})^k. \end{equation}

Suppose that we are able to find a single number $N_1$ for which we can prove $D_p(P_{N_1}) \le N_1^{\frac{1}{1000}}$.  Then Equation (\ref{submultk}) implies that $D_p(P_N) \le N^{\frac{1}{1000}}$ when $N$ is any power of $N_1$.  This implies the decoupling theorem, Theorem \ref{decpar}, with $\eps = \frac{1}{1000}$.  For any particular $N_1$, the decoupling constant $D_p(N_1)$ can be approximated to a given accuracy by a finite computation.  This isn't immediately obvious from the definition, but it isn't that difficult to show.  So in principle there exists a brute force proof of Theorem \ref{decpar} with $p=6$ (the most interesting $p$) and $\eps = \frac{1}{1000}$, where the proof is a giant finite computation to check that $D_6(P_{N_1}) \le N_1^\frac{1}{1000}$ for a particular $N_1$ together with Lemma \ref{submult}.

This situation is very different from the periodic Strichartz estimate, Theorem \ref{strichtorus}, or Vinogradov's mean value theorem, Theorem \ref{vino}.  For instance, suppose we somehow knew that Theorem \ref{strichtorus} holds when $d=3$ and $N = 10^{10}$.  Recall that Theorem \ref{strichtorus} is an $L^p$ estimate for periodic solutions to the Schrodinger equation with frequencies at most $N$.  If we somehow knew optimal bounds for periodic solutions with frequency at most $10^{10}$, I don't see how we could use that information to say anything about solutions with much larger frequencies, like $10^{1000}$.  

By switching our point of view from the original problem of periodic Strichartz estimates to the decoupling problem, we make it easier to combine information from different scales.
The real proof of the decoupling theorem does not involve a giant brute force computation like we described above.  It combines the multiscale idea from Lemma \ref{submult} with other ideas from the field, and we will discuss it more in the next section.

Next let's talk about the proof of Lemma \ref{submult}.  The proof is very short, and it illustrates how the statement of decoupling was crafted to combine information from different scales.

The first observation is that decoupling behaves in a nice way under translations and under linear changes of variable.  Suppose that $L: \RR^n \rightarrow \RR^n$ is a linear change of variables, or a translation, or a composition of those.  If we start with a decomposition $\Omega = \sqcup \theta$, then we get a new decomposition $L \Omega = \sqcup L \theta$.  The first observation is that the new decomposition has the same decoupling constant as the original one.

\begin{equation} \label{cov} D_p( L\Omega = \sqcup L \theta) = D_p( \Omega = \sqcup \theta). \end{equation}

\noindent If $g$ has Fourier support in $\Omega$ and $g = \sum g_\theta$, then we can perform a change of a variables to get a new function $\tilde g$ with Fourier support in $L \theta$.  Since the Fourier transform behaves in a nice way with respect to linear changes of variables and to translations, it's easy to track how the decoupling constant behaves and check (\ref{cov}).  

Now we start the proof sketch of Lemma \ref{submult}.  Suppose that $\hat f$ is supported in $\Omega$, the $(N_1 N_2)^{-2}$ neighborhood of $P$.  This neighborhood is divided into blocks $\theta$ of length $(N_1 N_2)^{-1}$, and we need to prove that

$$ \| f \|_{L^p}^2 \le D_p(P_{N_1})^2 D_p(P_{N_2})^2 \sum_\theta \| f_\theta \|_{L^p}^2. $$

We prove this bound in two steps.  Note that $\Omega$ is contained in the $N_1^{-2}$ neighborhood of $P$, which we can divide into blocks $\tau$ of length $N_1^{-1}$.  By definition of $D_p(P_{N_1})$, we have

$$ \| f \|_{L^p}^2 \le D_p(P_{N_1})^2 \sum_\tau \| f_\tau \|_{L^p}^2. \eqno{\textrm{(Step 1)}}$$

The support of $\hat f_\tau$ is contained in $\Omega \cap \tau$, which we can decompose as $\Omega \cap \tau = \sqcup_{\theta \subset \tau} \theta$. 

By the definition of $D_p$, 

$$ \| f_\tau \|_{L^p}^2 \le D_p( \Omega \cap \tau = \sqcup_{\theta \subset \tau} \theta)^2 \sum_{\theta \subset \tau} \| f_\theta \|_{L^p}^2. $$

Notice that there are $N_2$ different $\theta$ in each $\tau$.  In fact, there is a linear change of variables that takes $\Omega \cap \tau$ to the $N_2^{-1}$-neighborhood of $P$ and takes each $\theta$ to a block of length $N_2^{-1}$.  Therefore, $D_p( \Omega \cap \tau = \sqcup_{\theta \subset \tau} \theta) = D_p(P_{N_2})$.  Plugging in to the last indented equation, we get

$$ \| f_\tau \|_{L^p}^2 \le D_p(N_2)^2 \sum_{\theta \subset \tau} \| f_\theta \|_{L^p}^2. \eqno{\textrm{(Step 2)}}$$

Now if we combine Step 1 and Step 2, we get the desired inequality:

$$ \| f \|_{L^p}^2 \le D_p(P_{N_1})^2 \sum_\tau \| f_\tau \|_{L^p}^2 \le D_p(N_1)^2 D_p(N_2)^2 \sum_\theta \| f_\theta \|_{L^p}^2. $$

\section{Ideas of the proof}

In this section, we discuss some of the ideas in the proof of the decoupling theorem for the parabola, Theorem \ref{decpar}.  By now there are actually several proofs of Theorem \ref{decpar} (cf. \cite{BD}, \cite{Li1}, \cite{GMW}).  Each proof has some advantages.  We will focus on the original proof in \cite{BD}, but as we go we will try to highlight certain ideas that appear in all of the proofs.

Recall that $P$ is the truncated parabola in $\RR^2$.  We let $\Omega$ be the $N^{-2}$-neighborhood of $P$, and we decompose $\Omega$ into $N$ pieces $\theta$, which are each approximately rectangles of dimensions $N^{-2} \times N^{-1}$.  Suppose $\hat f$ is supported on $\Omega$ and decompose $f = \sum_\theta f_\theta$.  To help illustrate the ideas, we focus on the following corollary of Theorem \ref{decpar}. 

\begin{cor} \label{heavy} If $f = \sum_\theta f_\theta$ as in the last paragraph, and $\| f_\theta \|_{L^\infty(\RR^2)} \le 1$ for every $\theta$, then

$$ |U_{N/10}(f) \cap B_{N^2}| \le C_\eps N^{1 + \eps}. $$

\end{cor}

First let us give a little context for the numbers that appear in this bound.  By the triangle inequality, $|f(x)| \le \sum_{\theta} |f_\theta(x)| \le N$.  So $U_{N/10} (f)$ is the region where $|f(x)|$ is biggest. The bound in Corollary \ref{heavy} is sharp, as we can see from the following example.

\begin{exam} Let $f(x)$ be the exponential sum

$$ f(x) = \sum_{n=1}^N e^{2 \pi i \left( \frac{n}{N} x_1 + \frac{n^2}{N^2} x_2 \right)}. $$

Each $f_\theta$ is a single term in the sum, and so $\| f_\theta \|_{L^\infty} = 1$.  

We can check directly that $f(mN, 0) = N$ for any integer $m$ because each term in the sum is 1.  Also if $x$ lies in a ball of radius 1/100 around $(mN, 0)$, then each term in the sum has real part more than 1/2, and so $|f(x)| \ge N/2$.  Therefore, $U_{N/10}(f) \cap B_{N^2}$ contains $\sim N$ balls of $\sim 1$ and has measure $\gtrsim N$.  

\end{exam}

We will give a rough sketch of the proof of Corollary \ref{heavy}.
The proof of Corollary \ref{heavy} is simpler than the whole proof of Theorem \ref{decpar}, but it shows most of the main ideas.  

\subsection{Orthogonality} \label{subsecorth}

Under the hypotheses of Corollary \ref{heavy}, it may well happen that $|f_\theta(x)| \sim 1$ for every $x$ and every $\theta$.  To prove Corollary \ref{heavy}, we need to show that for most points $x \in B_{N^2}$, there is a lot of cancellation in the sum $f(x) = \sum_\theta f_\theta(x)$.  The most fundamental tool for proving cancellation in Fourier analysis is orthogonality.  Since the sets $\theta$ are disjoint, the functions $f_\theta$ are orthogonal, and so

$$ \int_{\RR^2} |f|^2 = \sum_\theta \int_{\RR^2} |f_\theta|^2. $$

\noindent The functions $f_\theta$ are exactly orthogonal on $\RR^2$.  They are also approximately orthogonal over any sufficiently large set.  
Since the distance between any two (non-adjacent) $\theta$'s is at least $1/N$, the functions $f_\theta$ are morally orthogonal on any ball of radius $N$.  The rough reason for this approximate orthogonality is the following.  Suppose $\omega_1 \in \theta_1$ and $\omega_2 \in \theta_2$.  We have to check that the functions $e^{2 \pi i \omega_1 x}$ and $e^{2 \pi i \omega_2 x}$ are approximately orthogonal on a ball $B_N(x_0)$.  The inner product of $e^{2 \pi i \omega_1 x}$ and $e^{2 \pi i \omega_2 x}$ on $B_N(x_0)$ is

$$ \int_{B_N(x_0)} e^{2 \pi i \omega_1 x}  \overline{e^{2 \pi i \omega_2 x}} dx =  \int_{B_N(x_0)} e^{2 \pi i (\omega_1 - \omega_2) x}  dx.$$

\noindent Since $|\omega_1 - \omega_2| \ge 1/N$, the function $e^{2 \pi i (\omega_1 - \omega_2) x}$ oscillates significantly on $B_N(x_0)$, which causes some cancellation in that integral.  This approximate argument suggests the following heuristic.

\begin{heur} \label{approxorth} (Approximate orthogonality) If $B$ is a square box of side length at least $N$, then

$$ \int_B |f|^2 dx \approx \sum_\theta \int_B |f_\theta|^2 dx. $$

\end{heur}

As written, this heuristic is not quite true, but there are more technical substitutes for it.  It is morally true, and it helps to imagine it in our proof sketch.  

By approximate orthogonality

$$ \int_{B_{N^2}} |f|^2 dx \approx \sum_\theta \int_{B_{N^2}} |f_\theta|^2 dx \le C N |B_{N^2}| = C N^{5}. $$

This gives an upper bound

\begin{equation} \label{bound1} |U_{N/10}(f) \cap B_{N^2}| \le C N^{3}. \end{equation}

To prove Corollary \ref{heavy}, we will have to improve the bound $N^3$ to $N^{1 + \eps}$.  

So far, we have only used that the rectangles $\theta$ are disjoint (and separated by at least $1/N$).  We will have to use more information about the $\theta$ in order to do better.  In fact, if the rectangles $\theta$ were laid out along a straight line, then bound (\ref{bound1}) would be best possible.  (We can see that by considering the exponential sum $f(x) = \sum_{n=1}^N e^{2 \pi i \frac{n}{N} x_1}$. )   To do better, we will have to take advantage of the way the rectangles $\theta$ follow the curve of the parabola.  In the next two subsections we set up some basic tools that will allow us to take advantage of the curvature of the parabola.

\subsection{Multiple scales} \label{subsecmultscales}

We want to study $f = \sum_\theta f_\theta$.  We can divide this sum into pieces in various ways.  If $M < N$, then we can cover $\Omega$ by $M$ rectangles $\tau$ of dimensions $M^{-1} \times M^{-2}$.  Imagine that $M$ divides $N$ so that each $\theta$ is contained in exactly one $\tau$.  Then we can write

$$f_\tau = \sum_{\theta \subset \tau} f_\theta, $$

$$f = \sum_\tau f_\tau. $$

In order to get better bounds for $f$, we will consider the functions $f_\tau$ at many different intermediate scales (many different choices of $M$).  

The number of $\theta \subset \tau$ is $\frac{N}{M}$, and so $|f_\tau(x)| \le \frac{N}{M}$.  We define $N_\tau = \frac{N}{M}$, which is the number of $\theta$ in $\tau$.

Since $|f(x)| \le \sum_\tau |f_\tau(x)|$, we see that if $|f(x)| \ge N/10$, then $|f_\tau(x)| \ge \frac{N_\tau}{20}$ for at least $M/20$ different $\tau$.  This suggests studying $U_{N_\tau/20} (f_\tau)$ for each $\tau$.  

We can use orthogonality (Lemma \ref{approxorth}) to bound $|U_\lambda (f_\tau)|$.  By itself, this will not lead to any new bounds.  In addition to that we will study the shape of $U_\lambda (f_\tau)$.  Because of their shapes, the sets $U_{N_\tau/20}(f_\tau)$ cannot overlap too much.  This geometric input will lead to improvement on the bound for $U_{N/10}(f)$.

\subsection{Wave packets} \label{subsecwavepack}

Suppose that $\tau \subset \RR^2$ is a rectangle.  Suppose that $\hat f_\tau$ is supported in $\tau$.  Then $f_\tau$ itself has a special geometric structure, which is called a wave packet decomposition.

This wave packet decomposition is based on a tiling of $\RR^2$ which is in some sense dual to $\tau$.  First, let $\tau^*$ be the dual rectangle.  If $\tau$ has dimensions $M^{-1} \times M^{-2}$, then $\tau^*$ would have dimensions $M \times M^2$.  The axis of $\tau^*$ with length $M^2$ corresponds to the axis of $\tau$ with length $M^{-2}$.  Next, let $\TT_\tau$ be a tiling of $\RR^2$ by rectangles congruent to $\tau^*$.

% Pictures

\begin{heur} \label{locconst} (Locally constant heuristic) If $\hat f_\tau$ is supported on a rectangle $\tau$ with center $\omega_\tau$, then for each rectangle $T \in \TT_\tau$,

$$ f_\tau(x) \approx a_T e^{2 \pi i \omega_\tau x}, $$

\noindent where $a_T \in \CC$ is a constant.  In particular, 

$$|f_\tau(x)| \textrm{ is approximately constant on each rectangle $T \in \TT_\tau$}. $$

\end{heur}

According to this heuristic, we can describe $f_\tau$ on all of $\RR^2$ in the form

\begin{equation} \label{wavepacketdec} f_\tau(x) \approx \sum_{T \in \TT_\tau} a_T e^{2 \pi i \omega_\tau x} \chi_T.   \end{equation}

\noindent Here $\chi_T$ is the characteristic function of $T$ (or a smoothed out version of it).  Each term on the right-hand side is called a wave packet, and the equation (\ref{wavepacketdec}) is called the wave packet decomposition of $f_\tau$.

This heuristic is again not quite literally true, but it can be replaced by more technical statements that are true.  It is morally true.

One origin of wave packet decompositions is particle-wave duality in quantum mechanics.  If $\hat f$ is supported in the parabola $P$, then $f$ satisfies the Schrodinger equation, which describes a quantum mechanical particle moving in a vacuum.  (Here we have $f(x_1, x_2)$, and we think of $x_2$ as the time variable $t$.)   Quantum mechanical particles can behave almost like classical particles for significant time periods.  A classical particle in a vacuum moves with constant velocity, tracing out a straight line in space time.  A single wave packet describes a quantum mechanical particle behaving almost classically.

% Picture.

Let us try to give some idea why the locally constant heuristic makes sense.  The Fourier transformation behaves in a nice way with respect to linear changes of variables and translations.  Because of this, it actually suffices to understand the wave packet decomposition when $\tau$ is the square $[-1,1]^2$.  Also the wave packet decomposition makes sense in any dimension, and the proofs are basically the same.  For simplicity, let us consider dimension 1.  Now we have a function $f: \RR \rightarrow \CC$ with $\hat f$ supported in $[-1,1]$.  The wave packet decomposition, Equation (\ref{wavepacketdec}), says that $f(x)$ is roughly constant on each unit interval.
This vague statement is closely related to the the Whittaker-Shannon-Nyquist interpolation theorem, which says that if $\hat f$ is supported in $[-1,1]$, then the whole function $f(x)$ can be recovered from the values $f(n/2)$, with $n \in \ZZ$.  Informally, this suggests that ``nothing significant is happening on length scales smaller than 1/2".   Here is another way to think about it.  Since $\hat f$ is supported in $[-1,1]$, 

$$ f(x) = \int_{-1}^1 \hat f(\omega) e^{2 \pi i \omega x} dx. $$

\noindent For $|\omega| \le 1$, each function $e^{2 \pi i \omega x}$ varies slowly, and looks roughly constant on any scale significantly smaller than 1.  The function $f$ itself is a linear combination of these slowly varying functions, and so we may hope that $f$ also looks roughly constant at scales smaller than 1.  

\subsection{Transversality} \label{subsectransver}

We are now ready to return to the proof sketch of Corollary \ref{heavy}.  By bringing into play the wave packet structure of $f_\tau$, we will see how to improve on the bound from subsection \ref{subsecorth}, which only used orthogonality.  At this point, the curvature of the parabola will come into play.  

Recall that each $\theta$ is an $N^{-2} \times N^{-1}$ rectangle in the $N^{-2}$-neighborhood of the truncated parabola $P$.  By the hypotheses of Corollary \ref{heavy}, we know that $\| f_\theta \|_{L^\infty(\RR^2)} \le 1$.  We want to bound $U_{N/10}(f) \cap B_{N^2}$.

As in subsection \ref{subsecmultscales}, set $M = N^{1/2}$, and cover the parabola with $M$ rectangles $\tau$ of dimensions $M^{-1} \times M^{-2}$.  Set $N_{\tau} = N/M = N^{1/2}$, and let us try to understand $U_{N_\tau/20}(f_\tau)$ for each $\tau$.  We know that $|f_\tau|$ is locally constant on translates of $\tau^*$, which have dimensions $M \times M^2 = N^{1/2} \times N$.  We can also use orthogonality to estimate $\int_{B_N} |f_\tau|^2 dx$ for each ball $B_N$ of radius $N$.  Putting together this information, we conclude that for each $B_N$, $U_{N_\tau/10}(f_\tau) \cap B_N$ is contained in $\lesssim 1$ translates of $\tau^*$.  In other words, on each $B_N$, each $f_\tau$ has only around 1 wave packet of amplitude $\sim N_\tau$.  

Now we are ready to take advantage of the curvature of the parabola.  Because of the curvature of $P$, the rectangles $\tau$ are oriented in different directions, and so the dual rectangles $\tau^*$ point in different directions.  
% Picture
On each $B_N$, $U_{N_\tau/10}(f_\tau)$ is essentially one translate of $\tau^*$.  Because all these rectangles point in different directions, they don't overlap very much.  The set $U_{N/10}(f)$ should lie in $U_{N_\tau/20}(f_\tau)$ for most $\tau$, and so $U_{N/10}(f) \cap B_N$ has to lie in a constant number of balls of radius $N^{1/2}$.  This geometric observation allows us to improve the bound for $U_{N/10}(f)$ beyond what we got from orthogonality alone.

The most effective way to study $U_{N/10}(f)$ on each of these balls of radius $N^{1/2}$ is to repeat the same method, using larger $\tau$'s with $M = N^{1/4}$.  Continuing in this way through many scales, we eventually see that $U_{N/10}(f) \cap B_N$ has to lie in at most $N^\eps$ balls of radius 1.  This gives an upper bound

\begin{equation} \label{bound2} |U_{N/10}(f) \cap B_{N^2}| \le C_\eps N^{2+\eps}. \end{equation}

We will call the argument in this section the orthogonality/transversality method, because those are the two main tools that go into it.  This argument is essentially due to Bennett-Carbery-Tao \cite{BCT}.  We will discuss their work more in section \ref{subsecmulti} below.  The orthogonality/transversality method improves on just orthogonality, but 
to prove Corollary \ref{heavy}, we will have to improve the bound $N^{2 + \eps}$ to $N^{1 + \eps}$.  

\subsection{Induction on scales and transversality together} \label{subsecindplustrans}

To get the sharp bound in Corollary \ref{heavy}, Bourgain and Demeter combined the ideas from the last subsection with induction on scales (as in Section \ref{secindscales}).  As in the last subsection, we set $M = N^{1/2}$ and cover the parabola with $M$ rectangles $\tau$ of dimensions $M^{-1} \times M^{-2}$.  In the orthogonality/transversality argument, we had to understand $U_{N_\tau/20}(f_\tau)$, and we controlled it with the following observation:

$$ \textrm{ (1) Local orthogonality gives an upper bound on $| U_{N_\tau/20}(f_\tau) \cap B_N |$ for each box $B_N$ of side length $N$.} $$

We can also bring into play induction on scales.  After a change of variables, estimating $|U_{N_\tau}(f_\tau)|$ is equivalent to our original problem, Corollary \ref{heavy}, but with $N^{1/2}$ rectangular tiles instead of $N$ tiles.  So we can also use induction on scales to bound $|U_{N_\tau}(f_\tau)|$.  

$$ \textrm{ (2) Induction on scales gives an upper bound on $|U_{N_\tau}(f_\tau) \cap B_{N^2}|$.} $$

The proof of decoupling in \cite{BD} uses (1) and (2) together.  Combining (1) and (2) leads to the sharp bound in Corollary \ref{heavy}

When I was first reading the proof of decoupling, I was surprised and even troubled that combining induction on scales with orthogonality/transversality is so powerful.  The orthogonality/transversality method gives an interesting but non-optimal bound.  Induction on scales by itself does not give any bound.  Why do these ingredients become so much stronger when we mix them together?

Initially, the argument even felt fishy to me.  Let's look back at points (1) and (2) above.  Why should we combine them?  If (2) is stronger than (1), then why not just use (2)?  If (1) is stronger than (2), then why not just use (1)?  I gradually realized that (1) and (2) give different types of information about $U_{N_\tau/20}(f_\tau)$.  Neither one is stronger than the other.  They are different and they give complementary information.

Induction on scales gives information about the total measure of $U_{N_\tau/20}(f_\tau)$ in $B_{N^2}$.  Local orthogonality also implies a bound on the total measure of $U_{N_\tau/20}(f_\tau)$ in $B_{N^2}$.  The bound on the total measure coming from induction is stronger than the bound coming from orthogonality.  But (1) is a local bound: it bounds $|U_{N_\tau/20}(f_\tau) \cap B_N|$ for each box of side $N$.  For a small box $B_N$ of side length $N$, the bound on $|U_{N_\tau/20}(f_\tau) \cap B_N|$ coming from (1) is stronger than the bound coming from (2).  Induction on scales controls the total measure of $U_{N_\tau/20}(f_\tau)$, and local orthogonality forces $U_{N_\tau/20}(f_\tau)$ to be rather spread out.

%Let us illustrate this with three pictures, which illustrate what the different tools tells us about $U_{N_\tau/20}(f_\tau)$:

%First.  Suppose we only knew local orthogonality (1).  Then $U_{N_\tau/20}(f_\tau)$ may look as follows:

%Picture 1.

%Second. Suppose we only knew induction on scales (2).  Now we have a stronger bound for $|U_{N_\tau/20}(f_\tau)|$, but we have no information about its shape.  Then $U_{N_\tau/20}(f_\tau)$ may look as follows:

%Picture 2.

%Finally, suppose we know both (1) and (2).  We now have the best of both worlds: the strong bound on the total measure from (2) plus the bounds on the local structure from (1).  Now $U_{N_\tau/20}(f_\tau)$ must look as follows:

%Picture 3.

To summarize, the bound (2) from induction gives the best information about the measere of $U_{N_\tau/20}(f_\tau)$.  But the bound (1) from local orthogonality gives us additional information about the shape of $U_{N_\tau/20}(f_\tau)$: in particular, for any box $B_N$ of side length $N$, $U_{N_\tau/20}(f_\tau) \cap B_N$ consists of at most a constant number of $N^{1/2} \times N$ rectangles.

Now we have digested the information that (1) and (2) give us about $f_\tau$, for each $\tau$.  The reader may wonder why information about the shape of $U_{N_\tau/20}(f_\tau)$ helps to bound the measure of $U_{N/10}(f)$.  The point is that it is difficult for different functions $f_{\tau_1}$ and $f_{\tau_2}$ to be large in the same place.  Notice that if $|f(x)| = |\sum_\tau f_\tau(x)|$ is large, then we must have $|f_\tau(x)|$ large for many different $\tau$ at the same point $x$.  If we knew (2) but not (1), it would be possible for $U_{N_\tau/20}(f_{\tau_1})$ and $U_{N_\tau/20}(f_{\tau_2})$ to be equal to each other.  But if we use (1) and (2) together, then we get a much stronger estimate for the measure of the intersection$U_{N_\tau/20}(f_{\tau_1}) \cap U_{N_\tau/20}(f_{\tau_2})$.

%Picture 4.

Here is another way to think about the leverage we get by adding induction on scales to the transversality/orthogonality argument from subsection \ref{subsectransver}.  Recall that we covered our original tiles $\theta$ with $M$ rectangles $\tau$ with dimensions $M^{-1} \times M^{-2}$, and we considered $f_\tau$.  In the argument from subsection \ref{subsectransver}, we started by picking $M = N^{1/2}$.  Continuing through the argument, we then used $M = N^{1/4}$, then $M = N^{1/8}$ and so on.  At each of these scales, we used the wave packet structure of the $f_\tau$ and we took advantage of transversality between the wave packets of the different $f_\tau$'s.

When we add in induction on scales, we are implicitly considering many different scales.  We started as before by using the scale $M = N^{1/2}$.  When we apply induction to a given $f_\tau$, and we unwind the induction, then we are really applying the same argument to $f_\tau$.  When we apply the argument to $f_\tau$, it gets decomposed as $f_\tau = \sum_\gamma f_\gamma$, where each $\gamma$ contains $N^{1/4}$ of the $\theta$ in $\tau$.  The total number of $\gamma$ covering all the different $\tau$'s is $M = N^{3/4}$, a scale that we never used in subsection \ref{subsectransver}.  If we fully unwind the inductive argument, it brings into play wave packets at every scale.  And it takes advantage of transversality between wave packets at every scale.  In some sense, the extra power comes from using transversality at every scale instead of just the special scales $M= N^{1/2}, N^{1/4}, N^{1/8},...$ which were used in subsection \ref{subsectransver}.

\subsection{Final comments}

As we mentioned earlier, there are a number of different proofs of decoupling.  In \cite{Li1}, Zane Li gave a new proof of Theorem \ref{decpar} based on Wooley's method of efficient congruencing (cf. \cite{Woo1}, \cite{Woo2}, \cite{Woo3}).  In \cite{GMW}, Maldague, Wang, and I gave a new proof of Theorem \ref{decpar} based on ideas from projection theory in geometric measure theory such as Orponen's work \cite{O}.  
One common feature of all these proofs is to bring into play $f_\tau$ with $\tau$ at every scale, and to take advantage of some type of transversality at every scale.  
%In most of the proofs, the appearance of all these scales is a little under the surface, because you have to unwind the induction in order to find most of the scales.

Vinogradov's work on the mean value conjecture \cite{V} already has this key feature: it uses $f_\tau$ for $\tau$ of every scale (after unwinding the induction) and takes advantage of some type of transversality at every scale.  Vinogradov's work \cite{V} is the first work I am aware of to take advantage of many scales of $\tau$ in estimating an exponential sum.  Within harmonic analysis, Wolff's work on local smoothing \cite{W4} used this key feature.  Bourgain's work on the restriction problem \cite{B1} took advantage of the transversality of wave packets of $f_\tau$ for a single scale of $\tau$.  Wolff's work \cite{W4} introduced a version of induction on scales which allowed him to take advantage of transversality of wave packets at every scale.  Using this method, he proved a decoupling theorem (for the cone) for large exponents $p$.  

The papers \cite{V} and \cite{W4} prove estimates for $|U_\lambda(f)|$ which are sharp when $\lambda$ takes the largest possible value, but not sharp for smaller $\lambda$.  For instance, the methods of \cite{V} or \cite{W4} could prove Corollary \ref{heavy}.  The advantage of Theorem \ref{decpar} is to give sharp estimates for $|U_\lambda(f)|$ for every $\lambda$.  For simplicity, we illustrated the method with $\lambda = N/10$, the largest possible value.  The same general method works for every value of $\lambda$, although there are some extra wrinkles in the argument.

\section{The Kakeya conjecture} \label{seckakeya}

In this section, we discuss why the restriction conjecture, Conjecture \ref{conjrestriction}, remains out of reach in dimension $n \ge 3$.  As we saw in the last section, Fourier analytic estimates in restriction theory are related to understanding how much rectangles pointing in different directions can overlap each other.  The Kakeya conjecture is a precise question about how much rectangles pointing in different directions can overlap each other.  (Actually, there are several related conjectures.)

Let us formulate the Kakeya conjecture in a way that connects with our discussion of wave packets.  Recall that $P \subset \RR^n$ denotes the truncated paraboloid:

$$ P = \{ \omega \in \RR^n : \omega_n = \sum_{j=1}^{n-1} \omega_j^2 \textrm{ and } 0 \le \omega_n \le 1 \}. $$

Cover $P$ with $N^{n-1}$ rectangular boxes $\theta$ of dimensions $\frac{1}{N} \times ... \times \frac{1}{N} \times \frac{1}{N^2}$.  For each $\theta$, let $\theta^*$ denote the dual box with dimensions $N \times  ... \times N \times N^2$.  The long direction of $\theta^*$ is equal to the short direction of $\theta$.  For each $\theta$, let $T_\theta$ denote a translate of $\theta^*$.

The tubes $T_\theta$ are related to wave packets that occur in the restriction problem.  In the restriction problem, we consider a function $f$ of the form 

\begin{equation} \label{deffP2} f(x) = \int_P a(\omega) e^{2 \pi i \omega x} d \mu_P(\omega). \end{equation}

\noindent The restriction problem asks to estimate $\| f \|_{L^p(\RR^n)}$ assuming that $|a(\omega)| \le 1$ for every $\omega$.  We can decompose $f$ as $f = \sum_\theta f_\theta$ where

\begin{equation}  f_\theta(x) = \int_{P \cap \theta} a(\omega) e^{2 \pi i \omega x} d \mu_P(\omega). \end{equation}

 Heuristically, each function $f_\theta$ is organized into wave packets, and in particular $|f_\theta|$ is locally constant on translates of $\theta^*$.  So the tubes $T_\theta$ correspond to wave packets of $f$.  Understanding how much the wave packets overlap helps to estimate $\| f \|_{L^p}$.

Now we are ready to formulate one version of the Kakeya conjecture.

\begin{conj} \label{conjkakeyavol} (Kakeya conjecture for volume) Suppose $n \ge 2$.  For each $\theta$ in the covering of $P \subset \RR^n$, let $T_\theta$ be a translate of $\theta^*$.  Then for each $\eps > 0$, 

$$ | \cup_\theta T_\theta | \ge C(n, \eps) N^{-\eps} \sum_\theta | T_\theta |. $$

\end{conj}

An argument of Fefferman \cite{F2} shows that the restriction conjecture implies the Kakeya conjecture.  If a set of tubes $\{ T_\theta \}$ is a counterexample to the Kakeya conjecture, we could build a counterexample to the restriction conjecture by choosing $f_\theta$ to concentrate on a single wave packet supported on $T_\theta$.  

Around 1920, Besicovitch constructed a remarkable example in 2 dimensions where $| \cup_\theta T_\theta | \sim \frac{1}{\log N} \sum_\theta | T_\theta |$.  Fefferman used this construction in \cite{F2} to give a counterexample to a cousin of the restriction conjecture called the ball multiplier problem.  

When $n=2$, Besicovitch's construction turns out to be tight: Davies proved that $| \cup_\theta T_\theta | \ge \frac{c}{\log N} \sum_\theta | T_\theta |$.  If $n \ge 3$, Besicovitch's construction still works, but we don't know good bounds in the other direction.  For example, if $n=3$, then Davies's method gives only 

$$| \cup_\theta T_\theta | \ge \frac{c}{N} \sum_\theta | T_\theta |. $$

Bourgain \cite{B1} improved the $\frac{c}{N}$ to $\frac{c}{N^{2/3}}$ and Wolff \cite{W2} improved it further to $\frac{c}{N^{1/2}}$.  At this point, it becomes very difficult to go further.  The best current bound is

$$| \cup_\theta T_\theta | \ge \frac{c}{N^{1/2 - \eps_0}} \sum_\theta | T_\theta |, $$

\noindent where $\eps_0$ is a small positive constant.  The proofs don't make $\eps_0$ explicit, but the best value given by current techniques is probably around $1/1000$.  This estimate was proven under an extra assumption by Katz-Laba-Tao \cite{KLT} and then proven in full generality by Katz-Zahl \cite{KZ}.  The arguments of \cite{B1} and \cite{W2} are fairly short, about five pages each, but the arguments of \cite{KLT} and \cite{KZ} are much more complex, about 50 pages each.

The reason that it is very difficult to improve on $\frac{c}{N^{1/2}}$ has to do with an "almost counterexample" which takes place in $\CC^3$.  This almost counterexample was first described in \cite{KLT}.  Consider the set

$$ H = \{ (z_1, z_2, z_3) \in \CC^3: |z_1|^2 + |z_2|^2 - |z_3|^2 = 1 \}. $$

\noindent This set is a 5-dimensional real manifold in $\CC^3$.  Its key feature is that it contains many complex lines.  Each point of $H$ lies in infinitely many complex lines contained in $H$.  Using this set $H$ as a guide, \cite{KLT} constructed a set of ``complex tubes" $T_j$ with ``dimensions" $N \times N \times N^2$, where $| \cup_j T_j | = \frac{c}{N^{1/2}} \sum_j |T_j|$.  These tubes overlap each other in a very intricate way.  They are complex tubes instead of real tubes, and they don't actually all point in different directions, but Wolff's argument from \cite{W2} does apply to them.  To beat the Kakeya estimate from \cite{W2}, one has to introduce into the argument some tool that rules out this "almost counterexample".  The papers \cite{KLT} and \cite{KZ} succeed in doing this, but the tools are much more complex and the quantitative bounds are rather weak.  It would be major progress in the field to give a good quantitative improvement to the Kakeya bound in \cite{W2}, let alone proving the Kakeya conjecture in full.

There is also a stronger version of the Kakeya conjecture which involves $L^p$ norms.  This version is important for the coming subsection.

\begin{conj} \label{conjkakeyaLp} (Kakeya conjecture for $L^p$ norms) Suppose $n \ge 2$.  For each $\theta$ in the covering of $P \subset \RR^n$, let $T_\theta$ be the characteristic function of translate of $\theta^*$, and let $T_{\theta, 0}$ be the characteristic function of $\theta^*$ itself.  The difference is that $\theta^*$ is centered at 0, but $T_\theta$ could have any center.  Then for any $\eps > 0$ and any $p$,

$$ \| \sum_\theta T_\theta \|_{L^p(\RR^n)} \le C(n, \eps) N^{\eps} \|  \sum_\theta T_{\theta,0} \|_{L^p(\RR^n)}. $$

\end{conj}

To digest this formula, notice that $\sum_\theta T_\theta(x)$ is the number of tubes through $x$.  The $p^{th}$ power of the left-hand side is $\int_\RR^n | \sum_\theta T_\theta(x)|^p dx$.  This is large if many points $x$ lie in many tubes from our set of tubes.  So the $L^p$ Kakeya conjecture says that not too many points $x$ can lie in many different tubes.

The restriction conjecture implies this stronger version of the Kakeya conjecture, which in turn implies the Kakeya conjecture for volumes, Conjecture \ref{conjkakeyavol}.  

Bourgain and Demeter proved a sharp decoupling theorem for the paraboloid $P \subset \RR^n$ for all $n$, which they used to give a sharp Strichartz estimate for tori in all dimensions, Theorem \ref{strichtorus}.  One reason this result came as a big surprise has to do with the Kakeya conjecture.  The proof of decoupling for the paraboloid involves estimating how much tubes pointing in different directions overlap.  When $n=2$, we know a great deal about how rectangles in different directions overlap, including the Kakeya conjecture for $n=2$.  But when $n \ge 3$, we don't know the Kakeya conjecture.  Although there was no formal connection between Kakeya and decoupling for the paraboloid, the Kakeya conjecture still made a sharp decoupling theorem in high dimensions seem out of reach, especially for an approach which is heavily based on estimating the overlaps of tubes pointing in different directions.  

\subsection{Multilinear Kakeya} \label{subsecmulti}

The Kakeya-type input into the proof of decoupling is called multilinear Kakeya.  It was formulated and proven by Bennett-Carbery-Tao \cite{BCT}.  Multilinear Kakeya is a cousin of Kakeya.  The setup is a little different, and we will explain it below, but it still gets at the idea that tubes pointing in different directions cannot overlap too much.  Remarkably, Bennett-Carbery-Tao proved sharp multilinear Kakeya estimates in all dimensions.  Their proof was simplified in \cite{Gu1} down to a few pages.

The multilinear Kakeya estimate in $\RR^n$ is an $L^p$ type estimate.  Suppose that $\ell_{j,a} \subset \RR^n$ is a line that makes a small angle with the $x_j$ axis (an angle at most $\frac{1}{100n}$ will do).  Let $T_{j,a}$ be the characteristic function of the unit neighorhood of $\ell_{j,a}$ -- the characteristic function of a tube.  Let $B_R \subset \RR^n$ denote a cube of side length $R$.

\begin{theorem} \label{multkak} (Multilinear Kakeya, \cite{BCT})

$$ \int_{B_R} \prod_{j=1}^n \left( \sum_{a = 1}^{A_j} T_{j,a} (x)  \right)^{\frac{1}{n-1}} dx \le C(n, \eps) R^\eps \prod_{j=1}^n A_j^{\frac{1}{n-1}}. $$

\end{theorem}

\noindent Let's take a moment to digest this estimate.  For a fixed $j$, think of the tubes $\{ T_{j,a} \}_{a=1}^{A_j}$ as tubes ``in direction $j$''.  Now $\sum_{a=1}^{A_j} T_{j,a}(x)$ is the number of tubes in direction $j$ going through $x$.  The integrand is $\prod_{j=1}^n \left( \sum_{a = 1}^{A_j} T_{j,a} (x)  \right)^{\frac{1}{n-1}}$, which is big if $x$ lies in many tubes from each direction.  So the integral on the LHS measures how many points $x$ lie in many tubes from each direction.  The multilinear Kakeya inequality says that there cannot be too many points which lie in many tubes from each direction.

The exponent $\frac{1}{n-1}$ makes the inequality sharp in two natural examples: the example when all the tubes go through the origin and an example when the tubes are arranged in a rectangular grid.  The exponent $\frac{1}{n-1}$ is the most important, and this bound implies sharp estimates with any other exponent.

It makes sense to compare Theorem \ref{multkak} with the $L^p$ Kakeya conjecture, Conjecture \ref{conjkakeyaLp}.  The main difference between them is that in the multilinear Kakeya theorem, the integrand is a product of $n$ factors, and we assume that the $n$ factors are transverse to each other in a strong sense.  The word multilinear refers to this product structure.

Theorem \ref{multkak} is also proven by induction on scales.  In the case that the tubes $T_{j,a}$ are exactly parallel to the $x_j$ axis (for all $j$ and $a$), Theorem \ref{multkak} reduces to the Loomis-Whitney inequality \cite{LW}, which we will recall a moment.  The general case of multilinear Kakeya is proven by applying Loomis-Whitney at many scales (cf. \cite{Gu1}).  The multilinear Kakeya inequality grew out of work by Bennett-Carbery-Wright on non-linear versions of the Loomis-Whitney inequality \cite{BC}.  

For completeness let us recall the statement of the Loomis-Whitney inequality.  One version is an inequality for integrals that looks reminiscent of Holder's inequality.  Suppose that $\pi_j: \RR^n \rightarrow \RR^{n-1}$ are projections onto the coordinate hyperplanes.  Then the Loomis-Whitney inequality says

$$ \int_{\RR^n} \prod_{j=1}^n f_j (\pi_j(x))^{\frac{1}{n-1}} dx \le \prod_{j=1}^n \| f_j \|_{L^1(\RR^{n-1})}^{\frac{1}{n-1}}. $$

\noindent There is a geometric corollary of this inequality which may feel more intuitive.  Suppose that $U \subset \RR^n$ is an open set, and that the projection of $U$ onto every coordinate hyperplane has $(n-1)$-volume at most $A$.  Then $U$ has $n$-volume at most $A^{\frac{n}{n-1}}$.  The case $n=2$ is straightforward, but the case $n=3$ is quite subtle.  It is one of my favorite problems to think through with students studying analysis.

When multilinear Kakeya was first proven, it seemed natural and remarkable, but it wasn't clear just how much impact it would have in restriction theory.  In \cite{BCT}, Bennett, Carbery, and Tao \cite{BCT} formulated and proved an interesting multilinear restriction conjecture.  They proved multilinear restriction by using multilinear Kakeya at many scales.  But it wasn't clear whether these multilinear estimates would lead to bounds on problems that were not multilinear, such as the original restriction conjecture.   

The paper \cite{BG} used these multilinear estimates to prove new partial results about the restriction problem.  It introduced a technique called the broad / narrow method which can sometimes reduce linear estimates to multilinear estimates.  

Remarkably, sharp decoupling theorems follow from multilinear Kakeya, even though there is nothing obviously multilinear about the statement of decoupling.  This was one of the big surprises in the development of the field.  The original Kakeya problem is much harder than multilinear Kakeya.  The original restriction problem is much harder than multilinear restriction.  There is also a multilinear version of decoupling.  A key fact that makes decoupling accessible is that the original decoupling problem is EQUIVALENT to multilinear decoupling.  This equivalence was noticed implicitly by Bourgain in \cite{B4}, and explicitly by Bourgain and Demeter in \cite{BD}.   Because of this connection between decoupling and multilinear decoupling, we can prove sharp estimates for the original decoupling problem using multilinear Kakeya, even though we don't know sharp estimates for the original Kakeya problem.

The connection between decoupling and multilinear decoupling is another important application of induction on scales.  It is based on the broad/narrow method.  Because of considerations of space, we don't give a detailed description here. %???

When multilinear Kakeya first appeared, it seemed like it might not have very many applications in harmonic analysis compared with the original Kakeya conjecture.  But now the situation has reversed: multilinear Kakeya currently has more applications in harmonic analysis than the original Kakeya conjecture would have even if we knew it.

\section{Applications of decoupling in harmonic analysis}

Decoupling theory has led to the solutions of several longstanding problems in harmonic analysis.  We give three examples here.  Each of these problems seemed out of reach a decade ago.

\subsection{The helical maximal function}
 
Hardy and Littlewood introduced their maximal function in the early 20th century.  The Hardy-Littlewood maximal function is based on averages over balls.  If $f: \RR^n \rightarrow \RR$, then
the average value of $f$ on the ball of radius $r$ around $x$ can be written as

$$ \frac{1}{|B_r|} \int_{B_r} f(x+y) dy. $$

\noindent The Hardy-Littlewood maximal function is defined by taking the supremum over $r$.

$$ Mf(x) = \sup_r  \frac{1}{|B_r|} \int_{B_r} |f(x+y)| dy. $$

\noindent Hardy and Littlewood proved that $\| M f \|_{L^p} \le C(p,n) \| f \|_{L^p}$ for all $p > 1$ but not for $p = 1$.

In the 1960s, Stein introduced a spherical maximal function \cite{Stsph}.  Suppose $f: \RR^n \rightarrow \RR$.  The average value of $f$ on the sphere of radius $r$ around $x$ can be written as

$$ \frac{1}{|S^{n-1}|} \int_{S^{n-1}} f(x + r \theta) d \theta. $$

\noindent The spherical maximal function is defined by taking the supremum over $r$:

\begin{equation} \label{sphmax} M_S f(x) := \sup_{r > 0} \frac{1}{|S^{n-1}|} \int_{S^{n-1}} |f(x + r \theta)| d \theta. \end{equation}

\noindent For $n \ge 3$, Stein proved that in $\RR^n$, $\| M_S f \|_{L^p} \le C(n,p) \| f \|_{L^p}$ for all $p > \frac{n}{n-1}$, but not for $p \le \frac{n}{n-1}$.  He conjectured that the same was true for $n=2$.  The case $n=2$ was proven by Bourgain in \cite{B8}.  

Stein's result was striking for the following reason.  A function $f \in L^p$ need only be defined almost everywhere.  It may be undefined or infinite on a lower-dimensional submanifold like a sphere.  So for a particular $x$ and $r$, the integral on the RHS of (\ref{sphmax}) may be infinite or undefined.  Nevertheless, if $f \in L^p$ for $p > \frac{n}{n-1}$, Stein showed that the spherical maximal function is actually defined for almost every $x$.  The curvature of the sphere is crucial in this estimate.  The spherical maximal function and the restriction conjecture were two fundamental connections between curvature and harmonic analysis that Stein investigated.

The spherical maximal function can be generalized by replacing the sphere by other curved submanifolds.  Many of the corresponding problems are still open.  After the sphere and circle, the next most fundamental case to look at is the case of the moment curve in $\RR^n$.  Here is the definition.  Consider the moment curve parametrized by $\gamma(t) = (t, t^2, t^3, ... , t^n)$.  We can build an averaging operator based on the moment curve as follows.  Suppose $f: \RR^n \rightarrow \RR$ and define

$$  A f(x) = \int_0^1 f( x + \gamma(t)) dt. $$

\noindent Geometrically, $A f(x)$ is the average value of $f$ on the translate of the moment curve starting at $x$.  Next we can consider different scalings of the moment curve.  Define

$$ A_r f(x) = \int_0^1 f (x + r \gamma(t)) dt. $$

\noindent Geometrically, $A_r f(x)$ is the average value of $f$ on a moment curve which has been scaled by a factor of $r$ and then translated to start at $x$.  Finally, we can define the helical maximal function by taking the maximum of these averages over different choices of $r$.

$$ M_{hel} f(x) := \sup_{r > 0} A_r f(x). $$

In analogy with the work of Stein and Bourgain on the circular maximal function, it is natural to ask when $\| M_{hel} f \|_{L^p(\RR^n)} \lesssim \| f \|_{L^p(\RR^n)}$.  
 In \cite{PS}, Pramanik and Seeger connected this problem (when $n=3$) to the decoupling problem for the cone, which Wolff had recently introduced in \cite{W4}.  In \cite{BD}, Bourgain and Demeter gave sharp estimates for the decoupling for the cone, but that by itself is not enough to give sharp estimates for the helical maximal function.  Recently, Ko-Lee-Oh \cite{KLO} and Beltran-Guo-Hickman-Seeger \cite{BGHS} independently proved the sharp $L^p$ estimate for the helical maximal function when $n=3$.
 
 \begin{theorem} (\cite{KLO} and \cite{BGHS}) For $p > 3$
  $\| M_{hel} f \|_{L^p(\RR^3)} \le C(p) \| f \|_{L^p(\RR^3)}$.
  
  If $p \le 3$, this estimate does not hold.
  
  \end{theorem}
  
  The case of higher dimensions remains open, although both groups have proven interesting estimates on helical averages in other dimensions as well.
 
\subsection{Pointwise convergence for the Schrodinger equation} \label{subseccarleson}

Consider the initial value problem for the linear Schrodinger equation in $\RR^d \times \RR$:

$$ \partial_t u(x,t) = i \triangle u(x,t), u(x, 0) = u_0(x). $$

\noindent We can write down the solution $u$ with the help of the Fourier transform.  If the initial data $u_0$ is rough, then the solution $u(x,t)$ will be rough also.  In this situation, $u(x,t)$ will solve the differential equation in a distributional sense, even if $u(x,t)$ is discontinuous.

Carleson \cite{Ca} raised the following problem.  

\begin{ques} What is the smallest $s$ so that whenever $u_0 \in H^s(\RR^d)$ and $u(x,t)$ is a distributional solution to the Schrodinger equation on $\RR^d \times \RR$ with initial data $u_0(x)$, then $\lim_{t \rightarrow 0} u(x,t) = u_0(x)$ for almost every $x \in \RR^d$?  
\end{ques}

\noindent This question helps to describe how regular distributional solutions to the Schrodinger equation are.  This question is actually a cousin of the restriction problem and the Strichartz estimate, although we will have to rewrite it a little bit to see how they are connected.

Because $u$ solves the Schrodinger equation, the space-time Fourier transform $\hat u$ is supported on the infinite paraboloid.  One has to prove some estimates about how badly $u(x,t)$ oscillates for small $t$.  After some standard arguments (scaling and Littlewood-Paley), one can reduce these estimates to the case that $\hat u$ is supported on the truncated paraboloid $P$ and normalize so that $\| u_0 \|_{L^2(\RR^d)}=1$.  Now consider $U_{\lambda} (u) \subset \RR^d \times \RR$.  The Strichartz estimates give sharp bounds for $| U_\lambda(u)|$ in terms of $\lambda$.  A small variation gives sharp estimates for $|U_\lambda(u) \cap [0,R]^{d+1}|$ in terms of $\lambda$ and $R$.  Now let $\Pi_{\RR^d}(x,t) = x$ be the projection from space-time to space.  Carleson's pointwise convergence problem is related to the following question about the size of $\Pi_{\RR^d} ( U_\lambda(u) )$:

\begin{ques} Suppose that $\hat u$ is supported on the truncated paraboloid $P$.  Let $u_0(x) = u(x,0)$, and suppose that $\| u_0 \|_{L^2(\RR^d)} = 1$.  For any given $\lambda$, $R$, estimate the maximum possible size of  $| \Pi_{\RR^d}(U_\lambda u \cap [0,R]^{d+1}) |$. 
\end{ques}

\noindent The key difference between this problem and the Strichartz inequality is we have to estimate the $d$-volume of the projection of $U_\lambda (u)$ instead of the $(d+1)$-volume of $U_\lambda(u)$ itself.  This general problem is still open.  However we do understand a special case, which is sufficient to resolve the pointwise convergence problem.  Here is the special case:

\begin{ques} Suppose that $\hat u$ is supported on the truncated paraboloid $P$.  Let $u_0(x) = u(x,0)$, and suppose that $\| u_0 \|_{L^2(\RR^d)} = 1$.  
Suppose that $| \Pi_{\RR^d}(U_\lambda u \cap [0,R]^{d+1}) | \ge c R^d$.  How big can $\lambda$ be?  
\end{ques}

As a first example, suppose that $u_0$ is a smooth bump function approximating a constant function on $[0,R]^d$.  Because $\|u_0\|_{L^2} = 1$, we have $|u_0(x)| \sim R^{-d/2}$ on most of $[0,R]^d$.  In this case, $u(x,t)$ is roughly constant on $[0,R]^{d+1}$, and so $\lambda$ is also $\sim R^{-d/2}$.  

This first example is not the worst case.  In case $d=1$, the worst case example was found by Dahlberg-Kenig \cite{DK}.  It is given when $u(x,t)$ is a single wave packet, essentially supported on a tilted rectangle with dimensions $R^{1/2} \times R$. 

% Picture!

\noindent In this case, $u_0(x)$ is essentially supported on an interval of length $R^{1/2}$, and so $|u_0(x)| \sim R^{-1/4}$ on this interval.  Then $|u(x,t)| \sim R^{-1/4}$ on the whole wave packet, and we get $\lambda \sim R^{-1/4}$.  Carleson \cite{Ca} had showed previously that this value of $\lambda$ is optimal.  This settles Carleson's problem in the case $d=1$, but the case of higher dimensions was open for 30+ years.

In higher dimensions, we can adapt the Dahlberg-Kenig example by taking many parallel wave packets with disjoint projections onto $\RR^d$.  This gives $\lambda = R^{-\frac{d}{2} + \frac{1}{4}}$.  
For a long time, it seemed plausible that this construction was sharp in any dimension.  In the last decade, mathematicians found other much more intricate examples.  The first was given by Bourgain \cite{B9} and there were several improvements leading up to \cite{B10} (cf. also \cite{LR}).  The last example gives $\lambda = R^{-\frac{d}{2} + \frac{d}{2d+2}}$.

This last example turns out to be sharp.  The case $d=2$ was proven in \cite{DGL} and the case of all $d$ was proven in \cite{DZ}.  Even for $d=2$, the proof in \cite{DZ} is simpler.  The key ingredient in these proofs is decoupling.  Decoupling is applied in a somewhat indirect way.  In particular, the proofs use decoupling many times at different scales.  

\begin{theorem} (\cite{DZ}, \cite{B10}) Suppose that $s > \frac{d}{2d+2}$.  If $u_0 \in H^s(\RR^d)$, and $u(x,t)$ is a (distributional) solution to the linear Schrodinger equation with initial data $u_0$.  Then $\lim_{t \rightarrow 0} u(x,t) = u_0(x)$ for almost every $x$.  

Suppose that $s < \frac{d}{2d+2}$.  There exists a function $u_0 \in H^s(\RR^d)$ with the following bad behavior.  Let $u(x,t)$ be the (distributional) solution to the linear Schrodinger equation with initial data $u_0$.  For this function, $\limsup_{t \rightarrow 0} |u(x,t)| = + \infty$ for almost every $x \in \RR^d$.  

\end{theorem}

\subsection{The local smoothing problem}

Wolff introduced decoupling in his work on the local smoothing problem \cite{W4}.  This problem is an estimate about solutions to the wave equation.

Suppose that $u(x,t)$ solves the wave equation $\partial_t^2 u = \triangle u$, with $x \in \RR^d$ and $t \in \RR$, and with initial data $u(x,0) = u_0(x)$ and $\partial_t u(x,0) = u_1(x)$.  The local smoothing problem concerns Sobolev-type bounds for the wave equation:  Given bounds on some Sobolev norms of $u_0$ and $u_1$, what bounds can we prove on the Sobolev norms of $u$?  

To make things simple and concrete, let us suppose that $u_1 = 0$ and that $\hat u_0$ is supported in a ball of radius $N$ in frequency space.  Then we would like to find all bounds of the form

$$ \| u(x,t) \|_{L^p(\RR^d \times [0,1])} \le C N^\alpha \| u_0(x) \|_{L^p(\RR^d)}. $$

The word local in local smoothing refers to the time interval $[0,1]$.  A global estimate would give a bound on $\RR^d \times \RR$, whereas a fixed-time estimate would give a bound for $\RR^d \times \{ t_0 \}$ for some fixed $t_0$ (such as $t_0 = 1$).  Global in time estimates, local in time estimates, and fixed time estimates are all interesting.  Sharp fixed time estimates were established by Peral \cite{P} and Miyachi \cite{M} around 1980.  The word `smoothing' in local smoothing is because the power of $\alpha$ in the local in time estimates is smaller than the power in a fixed time estimate.

In \cite{Sog}, Sogge formulated the local smoothing conjecture, and he proved the first local smoothing estimates improving upon the $\alpha$ given by the fixed time estimates.

\begin{conj} \label{locsmooth} (\cite{Sog}) Suppose $d \ge 2$.  Suppose that $u(x,t)$ solves the wave equation in $\RR^d \times \RR$, with initial data $u(x,0) = u_0(x)$ and $\partial_t u(x,0) = 0$.  Suppose that $\hat u_0$ is supported in the ball of radius $N$.  Then, if $2 \le p \le \frac{2d}{d-1}$, then 

$$ \| u(x,t) \|_{L^p(\RR^d \times [0,1])} \le C(d, \eps) N^\eps \| u_0 \|_{L^p(\RR^d)}. $$

If $p > \frac{2d}{d-1}$, then

$$ \| u(x,t) \|_{L^p(\RR^d \times [0,1])} \le C(d, \eps) N^{\frac{d-1}{2} - \frac{d}{p} + \eps} \| u_0 \|_{L^p(\RR^d)}. $$

\end{conj}

The case $p = \frac{2d}{d-1}$ is the critical exponent, and it implies all the other estimates for a given dimension $d$.  In \cite{W4}, Wolff introduced decoupling and used it to show that Conjecture \ref{locsmooth} holds when $d=2$ and $p > 74$.  Wolff also observed that the local smoothing conjecture in dimension $d$ implies the Kakeya conjecture in dimension $d$, by adapting Fefferman's argument from \cite{F2}.  Therefore, the full conjecture remains out of reach for all $d \ge 3$.

In \cite{BD}, Bourgain and Demeter proved a complete decoupling theorem for the cone.  This implies that Conjecture \ref{locsmooth} holds in $\RR^d$ for all $p > \frac{2(d+1)}{d-1}$.  In particular, when $d=2$, local smoothing holds for all $p > 6$.  When $d=2$, the critical exponent for local smoothing is $p=4$.  

In \cite{GWZ}, Wang, Zhang, and I proved the local smoothing conjecture when $d=2$ for $p=4$ (and hence for all $p$).  The proof of local smoothing does not use decoupling per se, but it is strongly influenced by the ideas in the proof of decoupling, including induction on scales.

\section{Frustrations, limitations, and open problems}

Decoupling and the ideas in the proof of decoupling have led to solutions of many problems that seemed out of reach a decade ago.  The proof is elegant in some ways.  In some ways, it feels like a proof `from the book'.  It is essentially self-contained and it is not that long.  But in other ways the proof is frustrating.  (Actually, there are now several proofs, and they have various advantages and disadvantages.  The community is actively trying to understand decoupling from different angles, and in five or ten years, we may have a different sense of the essential ingredients.)

In this section, I discuss some of my frustrations with the proof of decoupling, some limitations of the method, and some open problems.

\subsection{Too much induction} On one hand, induction on scales is the central idea in the proof of decoupling.  On the other hand, the heavy reliance on induction makes the proof difficult to read.  A lot of important stuff is happening inside the induction.  

For example, as we discussed in subsection \ref{subsecindplustrans}, I think that the leverage in the proof of decoupling comes from taking advantage of the transversality of wave packets of every scale, not just at a few scales.  For instance, suppose we cut the parabola $P$ into $M$ rectangles $\tau$ with $M = N^{5/16}$.  The proof of decoupling takes advantage of the transversality between the wave packets at this scale, but it is not easy to locate the place in the argument where this transversality is used because it is a little bit buried in the induction.  Even though I have thought through the proof many times, it took me a good while to locate where wave packets at this particular scale are used.

Reading through the full proof of decoupling for the paraboloid, we see many different tricks for taking advantage of induction on scales.  Loomis-Whitney is used at many scales to prove multilinear Kakeya.  Multilinear Kakeya is used at many scales in the argument in Subsection \ref{subsectransver}.  The key induction on scales is described in Subsection \ref{subsecindplustrans}.  Induction on scales is also used in a different way to go back and forth between multilinear estimates and the original linear estimates, as we discussed in Subsection \ref{subsecmulti}.  Finally, many applications of decoupling actually use decoupling many times at different scales, as in Subsection \ref{subseccarleson}.

We might look at this and feel that using multiple scales is a craft with many aspects.  But we might also start to get the feeling that this is too many different tricks, and that we should try to take advantage of many scales in a more systematic way.

\subsection{What does decoupling say about the shapes of superlevel sets?}

Decoupling gives an estimate for $\| f \|_{L^p}$ or for the measure of the superlevel sets $U_\lambda(f)$.  Besides the measure of the sets $U_\lambda(f)$, decoupling also seems to be connected to the shape of the superlevel sets $U_\lambda(f)$.  Looking back through our discussion in Subsections \ref{subsectransver} and \ref{subsecindplustrans}, the shape of $U_\lambda(f)$ plays an important role, even though the final estimate only concerns the measure of $U_\lambda(f)$.  In particular, during the argument, we make use of some information about $| U_\lambda(f_\tau) \cap B|$ for various balls $B$ and for various $\tau$.  This information roughly describes how much the set $U_\lambda(f)$ can concentrate in balls.  The shape of $U_\lambda(f)$ is also connected to some applications of decoupling, such as the work on Carleson's pointwise convergence problem discussed in Subsection \ref{subseccarleson}.

Perhaps the shape of $U_\lambda(f)$ should be a more central character in decoupling.  What is the full information about the shape of $U_\lambda(f)$ which the proof method of decoupling gives?
Unfortunately this question is quite vague.  There are many possible ways we could describe the shape of $U_\lambda(f)$, and it's not clear which language to use.  But it is possible that discussing the shape of $U_\lambda(f)$ systematically throughout the whole story might make the arguments clearer or even stronger...

Here is one question from the harmonic analysis literature that has to do with the shape of $U_\lambda(f)$.  We consider a measure $\mu$ supported on a large ball $B_R \subset \RR^n$ which obeys the Frostman condition

\begin{equation} \label{frostman} \mu( B_r(x) ) \le r^\alpha \end{equation}

\noindent Here $0 < \alpha < n$ is fixed.  

\begin{ques} As in the restriction problem or the Strichartz inequality, suppose that $f: \RR^n \rightarrow \CC$ is given by 

$$ f(x) = \int_P a(\omega) e^{2 \pi i \omega x} d \mu_P(\omega). $$

For a given $n$ and $\alpha$, what is the best exponent $\gamma$ in the inequality

$$ \| f \|_{L^2(d \mu)} \le C R^\gamma \| a \|_{L^2(P)}, $$

\noindent among all functions $f$ as above and all measures $\mu$ obeying the Frostman condition (\ref{frostman}) with exponent $\alpha$.

\end{ques}

In dimension $n=2$, this question is well understood for all $\alpha$ by work of Mattila and Wolff, cf. \cite{W6}.  But for $n \ge 3$, the problem is far from understood.  In \cite{DZ}, Du and Zhang gave a sharp answer for $\alpha = n-1$.  No other cases are fully understood.  The Du-Zhang estimate for $\alpha = n-1$ is closely related to the solution of Carleson's problem on pointwise convergence for solutions of the Schrodinger equation.  Decoupling and multilinear restriction are the essential tools in their approach, and they use decoupling at many different scales.  

How much can the method of decoupling tell us about other values of $\alpha$?  Is there anything fundamentally special about $\alpha = n-1$?  Also the Frostman condition (\ref{frostman}) can be replaced by other conditions, by replacing the function $r^\alpha$ by other functions of $r$.  This would lead to other kinds of estimates about the shape of $U_\lambda(f)$.

\subsection{Limitations of the information used in the proof}

In the statement of decoupling, we assume that $\hat f$ is supported in $\Omega$, and we try to bound $\| f \|_{L^p}$ in terms of some information about $\| f_\theta \|_{L^p}$ for all the $\theta$ in the decomposition of $\Omega$.  If we look through the proof and check where the hypothesis $\supp (\hat f) \subset \Omega$ is used, we find that it is used only in fairly simple ways.

In the course of the proof, we consider $f_\tau$ for many different rectangles $\tau$.  The proof relies crucially on two facts.  The first is the locally constant heuristic: 

\begin{equation} \label{locconst'}  \textrm{ For each $\tau$,  $|f_\tau|$ is approximately constant on each translate of $\tau*$.} \end{equation}

The second is the local orthogonality heuristic.  If $\tau$ is a rectangle, and $\gamma$ are smaller rectangles contained in $\tau$, and if non-adjacent $\gamma$ are separated by at least $s$, then 

\begin{equation} \label{locorth} \int_{B} |f_\tau|^2 \approx \sum_{\gamma \subset \tau} \int_B |f_\gamma|^2, \end{equation}

\noindent whenever $B$ is a cube whose side length is longer than $s^{-1}$.  (The proof of decoupling also involves some linear changes of variables.  A ball $B$ in the new variables might correspond to an ellipsoid in the original variables.)

The Fourier support properties of the different functions $f$, $f_\tau$, $f_\theta$ are only really used to justify these two heuristics.  These two heuristics are consequences of the Fourier support hypotheses, but they don't encode all the information given by the Fourier support hypotheses.

This raises the question: which theorems of restriction theory can we prove only using the locally constant heuristic and local orthogonality?  Which theorems require us to use the Fourier support hypothesis in some other way?

The proofs of the different decoupling theorems essentially only use these two properties.  (I say essentially because some of the proofs also involve some pigeonholing of wave packets.)  Also, the strongest current work on the restriction conjecture only uses these two properties.  It's possible that the full restriction conjecture might follow only using these two properties.

 In restriction theory there are currently very few examples of techniques for exploiting the Fourier support of $f$ that use Fourier support information in some other way.  (One example is to use the even integer trick, Lemma \ref{eventrick}, together with number theoretic input.  An interesting recent example of this approach is the work on Strichartz-type estimates for the periodic Airy equation by Hughes-Wooley \cite{HW}.). 

However, there are a number of problems in restriction theory where I strongly doubt that these two properties are sufficient to give full answers.  One example is the the problem of estimating the $L^p$ norms of the functions

$$  f_{k,N}(x) = \sum_{a = 1}^N e^{2 \pi i a^k x} .$$

\noindent As we discussed in Section \ref{analnumb}, the $L^p$ norms of $f_{k,N}$ are well understood for $k=2$ and wide open for $k \ge 3$.  When $k=2$, the different proofs all use some information besides the locally constant heuristic and local orthogonality.  I believe the sharp estimates for $k=2$ cannot be proven by an argument using only those two properties.

There is an interesting generalization of this $L^p$ problem which I think is a good test case for going beyond the locally constant property and local orthogonality.  As we mentioned in Section \ref{analnumb}, $\| f_{2,N} \|_{L^4([0,1])} \le C_\eps N^{1/2 + \eps}$.  

\begin{ques} \label{genf2N} We consider a sequence of frequencies $\omega_a$, with $a = 1, ... N$, which behave approximately like the squares $a^2$, in the sense that

$$ \omega_{a+1} - \omega_a \sim a, $$

$$ (\omega_{a+1} - \omega_a) - (\omega_{a} - \omega_{a-1}) \sim 1. $$

For such a choice of frequencies $\omega_a$, define

$$ f(x) = \sum_{a=1}^N e^{2 \pi i \omega_a x}. $$

Estimate $\| f \|_{L^4([0,1])}$.  Is it true that $\| f \|_{L^4([0,1])} \le C_\eps N^{1/2 + \eps}$?

\end{ques}

 As far as I know, it is possible that $\| f \|_{L^4([0,1])} \le C_\eps N^{1/2+\eps}$ in this much more general setting.  However, the proofs that work for $f_{2,N}$ do not generalize to this setting.  And the method of decoupling can prove only limited things.  In \cite{FGM}, Fu, Maldague, and I explored how much we can say about this question using ideas of decoupling theory.  As part of that investigation, we explain the version of the locally constant property which appears in this setting, which goes back to Bourgain's work \cite{B12} on Montgomery's conjecture. The main theorems of \cite{FGM} give sharp $L^p$ estimates for much shorter sums: sums of length $\sim N^{1/2}$.  For these shorter sums, the locally constant property and the methods of decoupling are effective.  But for longer sums, they seem much less effective, and I believe that some different tools are needed.

Question \ref{genf2N} is also related to a question of Erd{\H o}s about sumsets of convex sets.  A sequence $\omega_a$ is called convex if $ (\omega_{a+1} - \omega_a) - (\omega_{a} - \omega_{a-1}) > 0$ for all $a$.  Notice that the set of frequencies in Question \ref{genf2N} is a convex sequence.  If $A$ is a convex sequence, then Erd{\H o}s conjectured that $|A+A| \ge c_\eps |A|^{2 - \eps}$.  Here $A+A$ denotes all sums of two elements of $A$.  This conjecture is open.  There is interesting recent work on it by Schoen and Shkredov, \cite{Sh}, who proved that $|A + A| \ge c_\eps  |A|^{1.6 - \eps}$.  This beats the previous best estimate $|A|^{1.5}$, which had stood for a long time.  If $A$ denotes the frequencies in Question \ref{genf2N}, and if indeed $\| f \|_{L^4([0,1])} \le C_\eps N^{1/2 + \eps}$, then it would follow that $|A + A| \ge c_\eps |A|^{2 - \eps}$.  The best bound I could prove using the methods of decoupling gives $|A+A| \ge c |A|^{1.5}$.  Work in combinatorics such as \cite{Sh} may give clues on how to go further in problems like Question \ref{genf2N}.  

We can also ask an analogous question about sequences of frequencies that behave roughly like $k^{th}$ powers.

\begin{ques} \label{genfkN}
Suppose that the function $\phi: [0,N] \rightarrow \RR$ behaves approximately like the function $t^k$ in the following sense.  Let $\phi^{(j)}$ denote the $j^{th}$ derivative of $\phi$.

\begin{itemize}

\item $0 = \phi(0) = \phi'(0) = ... = \phi^{(k-1)}(0)$.

\item $t^{k-1} \le \phi^{(k)}(t) \le 2 t^{k-1}$.

\end{itemize}

Let $\omega_a = \phi(a)$, and define

$$ f(x) = \sum_{a=1}^N e^{2 \pi i \omega_a x}. $$

Estimate $\| f \|_{L^{2k}([0,1])}$.  Is it true that $\| f \|_{L^{2k}([0,1])} \le C_\eps N^{1/2 + \eps}$?

\end{ques}

If the answer to this question is affirmative, then it would imply Conjecture \ref{HypK*}, which says that $\| f_{k,N} \|_{L^{2k}} \le C_\eps N^{1/2 + \eps}$, where $f_{k,N} = \sum_{a = 1}^N e^{2 \pi i a^k x}$.  An affirmative answer would show that not only this conjecture is true but the estimate is quite robust, and only depends on some basic analytic features of the sequence of frequencies $a^k$.  The proof of the Vinogradov mean value conjecture, Theorem \ref{vino}, is robust in this sense: it applies not just to the sequence of frequencies $(a, a^2, ..., a^k)$, but also to any sequence of frequencies with similar basic analytic features.    If the answer to Question \ref{genfkN} is negative, it will show that Conjecture \ref{HypK*}, if true, is not as robust as Theorem \ref{vino}.  This would show that any proof of Conjecture \ref{HypK*} must use some finer properties of the sequence $a^k$.  This would help to clarify the nature of the difficulty in this old problem.

Question \ref{genfkN} has barely been investigated, but it is related to some interesting recent work of Hanson-Roche-Newton-Rudnev \cite{HRR} on higher convexity and iterated sumsets.

\end{document}